\documentclass[12pt]{amsart}
\allowdisplaybreaks
\usepackage[english]{babel}
\usepackage[pdftex,paper=a4paper,portrait=true,textwidth=450pt,textheight=675pt,tmargin=3cm,marginratio=1:1]{geometry}
\usepackage{amsfonts}
\usepackage[dvips]{graphics}
\usepackage[colorinlistoftodos]{todonotes}
\usepackage{amsmath}
\usepackage{amsthm}
\usepackage{amssymb}
\usepackage{bbm}
\usepackage{comment}
\usepackage{cancel}
\usepackage{color}
\usepackage[curve]{xypic}
\usepackage{graphicx}
\newtheorem{theorem}{Theorem}[section]
\newtheorem{corollary}[theorem]{Corollary}
\newtheorem{proposition}[theorem]{Proposition}

\newtheorem{conjecture}[theorem]{Conjecture}

\theoremstyle{definition}

\newtheorem{remark}[theorem]{Remark}

\theoremstyle{property}

\DeclareFontFamily{OT1}{rsfs}{}
\DeclareFontShape{OT1}{rsfs}{n}{it}{<-> rsfs10}{}
\DeclareMathAlphabet{\curly}{OT1}{rsfs}{n}{it}

\newcommand\I{\mathcal I}

\renewcommand\O{\mathcal O}
\newcommand\PP{\mathbb P}

\newcommand\EE{\mathbb E}

\newcommand\F{\mathcal F}

\newcommand\E{\mathbb E}
\newcommand\C{\mathbb C}

\newcommand\Q{\mathbb Q}

\newcommand\Z{\mathbb Z}
\newcommand\cZ{\mathcal Z}

\newcommand\bsa{\boldsymbol{a}}
\newcommand\bsn{\boldsymbol{n}}

\renewcommand\t{\mathfrak t}

\newcommand\SU{\mathrm{SU}}

\newcommand\vd{\mathrm{vd}}
\newcommand\pt{\mathrm{pt}}
\newcommand\vir{\mathrm{vir}}

\newcommand\SW{\mathrm{SW}}

\newcommand\td{\mathrm{td}}
\newcommand\frakT{\mathfrak{T}}
\newcommand\rk{\operatorname{rk}}

\newcommand\ch{\operatorname{ch}}

\renewcommand\hom{\mathcal{H}{\it{om}}}

\newcommand\Pic{\operatorname{Pic}}

\newcommand\INTO{\ar@{^{(}->}[r]}

\newcommand{\smfr}[2]{\hbox{$\frac{#1}{#2}$}}

\DeclareRobustCommand{\SkipTocEntry}[4]{}
\def\eps{\varepsilon}

\begin{document}
\title[Virtual Segre and Verlinde numbers]{Virtual Segre and Verlinde numbers \\ of projective surfaces}
\author[G\"ottsche and Kool]{L.~G\"ottsche and M.~Kool}
\maketitle

\vspace{-1cm}

\begin{abstract}
Recently, Marian-Oprea-Pandharipande established (a generalization of) Lehn's conjecture for Segre numbers associated to Hilbert schemes of points on surfaces. Extending work of Johnson, 
they provided a conjectural correspondence between Segre and Verlinde numbers. For surfaces with holomorphic 2-form, we propose conjectural generalizations of their results to moduli spaces of stable sheaves of any rank. 

Using Mochizuki's formula, we derive a universal function which expresses virtual Segre and Verlinde numbers of surfaces with holomorphic 2-form in terms of Seiberg-Witten invariants and intersection numbers on products of Hilbert schemes of points. We prove that certain canonical virtual Segre and Verlinde numbers of general type surfaces are topological invariants and we verify our conjectures in examples. 

The power series in our conjectures are algebraic functions, for which we find expressions in several cases and which are permuted under certain Galois actions. Our conjectures imply an algebraic analog of the Mari\~{n}o-Moore conjecture for higher rank Donaldson invariants. For ranks $3$ and $4$, we obtain explicit expressions for Donaldson invariants in terms of Seiberg-Witten invariants.
\end{abstract}
\thispagestyle{empty}

\section{Introduction} 
\addtocontents{toc}{\protect\setcounter{tocdepth}{1}}

\subsection*{Segre numbers}

Let $S$ be a smooth projective surface over $\C$ and $\alpha \in K(S)$ a class in the Grothendieck group of coherent sheaves on $S$.
On the Hilbert scheme $S^{[n]}$ of $n$ points on $S$, we have the tautological class
\begin{equation*} 
\alpha^{[n]} := p_! (q^* \alpha) \in K(S^{[n]}),
\end{equation*}
where $q : \mathcal{Z} \rightarrow S$ and $p : \cZ \rightarrow S^{[n]}$ denote projections from the universal subscheme $\cZ \subset S \times S^{[n]}$, and $p_! := \sum_i (-1)^i R^i p_*$. We refer to the coefficients of the following generating series as Segre numbers of $S$
$$
\sum_{n=0}^{\infty} z^n \int_{S^{[n]}} c(\alpha^{[n]}) \in \Q[[z]],
$$
where $c$ denotes total Chern class. For $\alpha = -V$, where $V$ is the class of a vector bundle, $c(\alpha^{[n]}) = s(V^{[n]})$, where $s$ denotes total Segre class. Throughout the paper, we abbreviate $K:=K_S$ and $\chi := \chi(\O_S)$. The following theorem was proved in \cite{MOP3}.
\begin{theorem}[Marian-Oprea-Pandharipande] \label{MOPthm1}
For any $s \in \Z$, there exist $V_s$, $W_s$, $X_s$, $Y_s$, $Z_s \in \Q[[z]]$ with the following properties. For any $K$-theory class $\alpha$ of rank $s$ on a smooth projective surface $S$, we have
$$
\sum_{n=0}^{\infty} z^n \int_{S^{[n]}} c(\alpha^{[n]}) = V_s^{c_2(\alpha)} \, W_s^{c_1(\alpha)^2} \, X_s^{\chi} \, Y_s^{c_1(\alpha) K} \, Z_s^{K^2}.
$$
Moreover, under the formal change of variables $z = t(1+(1-s)t)^{1-s}$, we have 
\begin{align*}
V_s(z) &= (1+(1-s)t)^{1-s} (1+(2-s)t)^s, \\ 
W_s(z) &= (1+(1-s)t)^{\frac{1}{2}s-1} (1+(2-s)t)^{\frac{1}{2}(1-s)}, \\
X_s(z) &= (1+(1-s)t)^{\frac{1}{2} s^2-s} (1+(2-s)t)^{-\frac{1}{2}s^2+\frac{1}{2}}(1+(2-s)(1-s)t)^{-\frac{1}{2}}.
\end{align*}
\end{theorem}

The existence of the universal power series $V_s$, $W_s$, $X_s$, $Y_s$, $Z_s$ follows from \cite{EGL} and the formulae for $V_s$, $W_s$, $X_s$ were found (and proved to hold) by A.~Marian, D.~Oprea, and R.~Pandharipande \cite{MOP3}. Segre numbers have a rich history. When $S \subset \PP^{3n-2}$ and $L \cong \O(1)|_S$, the Segre number $\int_{S^{[n]}} s(L^{[n]})$ counts the number of $(n-2)$-dimensional projective linear subspaces of $\PP^{3n-2}$ that are $n$-secant to $S$. M.~Lehn's conjecture from 1999 \cite{Leh} gives formulae for $(V_{-1}W_{-1})$, $X_{-1}$, $Y_{-1}$, $Z_{-1}$. 
Lehn's conjecture was proved for $K$-trivial surfaces by Marian-Oprea-Pandharipande \cite{MOP1} and established in general in \cite{MOP2} building on \cite{MOP1} and work of C.~Voisin \cite{Voi}. Explicit expressions for $Y_s$, $Z_s$ are known for $s \in \{-2,-1,0,1,2\}$ \cite{MOP3} and proved except for the case of $Y_0$. 
Conjecturally, $Y_s, Z_s$ are algebraic functions for all $s \in \Z$ \cite[Conj.~1]{MOP3}.

\subsection*{Verlinde numbers} 

A line bundle $L$ on $S$ induces a line bundle $L_n$ on the symmetric product $S^{(n)}$ by $\mathfrak{S}_n$-equivariant push-forward of $L \boxtimes \cdots \boxtimes L$ along the morphism $S^{n} \rightarrow S^{(n)} = S^n / \mathfrak{S}_n$. Its pull-back along the Hilbert-Chow morphism $S^{[n]} \rightarrow S^{(n)}$ is denoted by $\mu(L)$. Together with
$E:=\det \O_S^{[n]}$
the lines bundles $\mu(L)$ generate the Picard group of $S^{[n]}$. Holomorphic Euler characteristics of line bundles on $S^{[n]}$ are known as Verlinde numbers of $S$. 
For any $r \in \Z$, we form a generating series 
$$
\sum_{n=0}^{\infty} w^n \, \chi(S^{[n]}, \mu(L) \otimes E^{\otimes r}).
$$
Similar to the previous section, Verlinde numbers are given by universal power series and the ones determined by $K$-trivial surfaces are known.
\begin{theorem}[Ellingsrud-G\"ottsche-Lehn] \label{EGLthm}
For any $r \in \Z$, there exist $g_r, f_r, A_r, B_r \in \Q[[w]]$ with the following properties. For any line bundle $L$ on a smooth projective surface $S$, we have
$$
\sum_{n=0}^{\infty} w^n \, \chi(S^{[n]}, \mu(L) \otimes E^{\otimes r}) = g_r^{\chi(L)} \, f_r^{\frac{1}{2} \chi} \, A_r^{L K} \, B_r^{K^2}.
$$
Moreover, under the formal change of variables $w = v(1+v)^{r^2-1}$, we have
\begin{align*}
g_r(w) = 1+v, \quad f_r(w) = (1+v)^{r^2}(1+ r^2 v)^{-1}.
\end{align*}
\end{theorem}
In \cite{EGL} $g_r(w)$, $f_r(w)$ were written in a different way. The compact form used here was first given in \cite{Joh}.
Serre duality implies $A_r = B_{-r} / B_r$ for all $r$. Furthermore, $A_r=B_r = 1$ for $r = 0,\pm 1$ \cite{EGL}. 
As in the case of Segre numbers, a general formula for $A_r, B_r$ is unknown. 

\subsection*{Segre-Verlinde correspondence}

From Theorems \ref{MOPthm1} and \ref{EGLthm}, we deduce
\begin{align*}
f_{r}(w) = W_s(z)^{-4s} \, X_s(z)^2, \quad g_{r}(w) = V_s(z) \, W_s(z)^2, 
\end{align*}
where $s =1+r$ and 
\begin{equation} \label{varchange1}
w = v(1+v)^{r^2-1}, \quad z = t(1+(1-s)t)^{1-s}, \quad v = t(1-rt)^{-1}.
\end{equation} 
The following conjecture was proposed in \cite{MOP3} based on work of D.~Johnson \cite{Joh}, which in turn was motivated by strange duality.
\begin{conjecture}[Johnson, Marian-Oprea-Pandharipande]
For any $r \in \Z$, $s=1+r$, and under the formal variable change \eqref{varchange1}, we have 
$$
A_{r}(w) = W_s(z) \, Y_s(z), \quad  B_{r}(w) = Z_s(z). 
$$
\end{conjecture}

\subsection*{Virtual Segre numbers}

Let $(S,H)$ be a smooth polarized surface. For any $\rho \in \Z_{>0}$, $c_1 \in H^2(S,\Z)$ algebraic, and $c_2 \in H^4(S,\Z)$, we denote by $M:=M_S^H(\rho,c_1,c_2)$ the coarse moduli scheme of rank $\rho$ Gieseker $H$-semistable torsion free sheaves on $S$ with Chern classes $c_1,c_2$. We assume $M$ contains no strictly semistable sheaves (with respect to the polarization $H$). For this introduction, we also assume there exists a universal sheaf $\E$ on $S \times M$, but we remove this assumption in Remark \ref{univdrop1}. We denote by $\pi_S : S \times M \rightarrow S$ and $\pi_M : S \times M \rightarrow M$ the projections to the factors. 
Consider the slant product
$$
/ : H^p(S \times M,\Q) \times H_q(S,\Q) \rightarrow H^{p-q}(M,\Q).
$$
The $\mu$-classes in Donaldson theory are defined as follows. For any  $\sigma \in H^k(S,\Q)$ 
\begin{equation*} 
\mu(\sigma):=  \Big(c_2(\E)-\frac{\rho-1}{2\rho} c_1(\E)^2 \Big) / \mathrm{PD(\sigma)} \in H^{k}(M,\Q),
\end{equation*}
where $\mathrm{PD(\sigma})$ denotes the Poincar\'e dual of $\sigma$. Formally, we can write 
$$
2\rho c_2(\E)-(\rho-1) c_1(\E)^2 =- 2\rho \ch_2(\E\otimes \det(\E)^{-\frac{1}{\rho}}),
$$ 
which shows that $\mu(\sigma)$ is independent of the choice of universal sheaf. For any class $\alpha \in K(S)$, we define 
$$
\ch(\alpha_M) :=  \ch(- \pi_{M !} ( \pi_S^* \alpha \cdot \E \cdot  \det(\E)^{-\frac{1}{\rho}})) \in A^*(M)_{\mathbb{Q}},
$$
where $A^*(M)_{\Q}$ denotes the Chow ring with rational coefficients. When the root $\det(\E)^{-1 / \rho}$ does not exist, the right hand side is defined by a formal application of the Grothendieck-Riemann-Roch formula. This factor ensures $\ch(\alpha_M)$ is independent of the choice of universal sheaf. We write $c(\alpha_M)=\sum_i c_i(\alpha_M)$ for the Chern classes corresponding to $\ch(\alpha_M)$.
When $b_1(S) = 0$ and $M:=M_S^H(1,0,n) \cong S^{[n]}$, we have $c(\alpha_M) = c(\alpha^{[n]})$.

The moduli space $M$ carries a virtual class constructed by T.~Mochizuki \cite{Moc}
\begin{align} 
\begin{split}\label{vd}
[M]^{\vir} \in H_{2\vd(M)}(M), \quad \vd(M) := 2\rho c_2 - (\rho-1) c_1^2 - (\rho^2-1) \chi.
\end{split}
\end{align}
We write $\pt \in H^4(S,\Z)$ for the Poincar\'e dual of the point class, $u$ is an extra formal variable, $\eps_\rho := \exp(2 \pi i /\rho)$ where $i = \sqrt{-1}$, and $[n]:=\{1, \ldots, n\}$ for any $n \in \mathbb{Z}_{\geq 0}$. For any (possibly empty) $J \subset [n]$, we write $|J|$ for its cardinality and $\|J\| := \sum_{j \in J} j$.
\begin{conjecture} \label{conj1}
Let $\rho  \in \Z_{>0}$ and $s \in \Z$. There exist  $V_s$, $W_s$, $X_s$, $Q_s$, $R_s$, $T_s \in \C[[z]]$, $Y_{J,s}$, $Z_{J,s}$, $S_{J,s} \in \C[[z^{\frac{1}{2}}]]$ for all $J \subset [\rho-1]$ with the following property.\footnote{These universal power series depend on $\rho$ and $s$. We suppress the dependence on $\rho$.} Let $S$ be a minimal surface of general type with $p_g(S)>0$ and $b_1(S) = 0$. Suppose $M:=M_S^H(\rho,c_1,c_2)$ contains no strictly semistable sheaves. For any $\alpha \in K(S)$ with $\rk(\alpha)=s$ and $L \in \Pic(S)$ the virtual Segre number $\int_{[M]^{\vir}} c( \alpha_M ) \, \exp(\mu(L)  + \mu(\pt) u)$ is the coefficient of $z^{\frac{1}{2} \vd(M)}$ of
\begin{align*}
\rho^{2 - \chi+K^2} \, V_s^{c_2(\alpha)} \, W_s^{c_1(\alpha)^2} \, X_s^{\chi} \, e^{L^2 Q_s + (c_1(\alpha)L) R_s + u \, T_s} \sum_{J \subset [\rho-1]} (-1)^{|J| \chi}  \, \eps_{\rho}^{\|J\| K c_1} \, Y_{J,s}^{c_1(\alpha) K} \, Z_{J,s}^{K^2} \, e^{(K L) S_{J,s}}.
\end{align*}
Moreover, under the formal change of variables $z = t (1+(1-\tfrac{s}{\rho}) t)^{1-\frac{s}{\rho}}$, we have
\begin{align*}
V_s(z) &= (1+(1-\tfrac{s}{\rho})t)^{1-s} (1+(2-\tfrac{s}{\rho})t)^s (1+(1-\tfrac{s}{\rho})t)^{\rho-1} , \\ 
W_s(z) &= (1+(1-\tfrac{s}{\rho})t)^{\frac{1}{2}s-1} (1+(2-\tfrac{s}{\rho})t)^{\frac{1}{2}(1-s)} (1+(1-\tfrac{s}{\rho})t)^{\frac{1}{2} - \frac{1}{2} \rho}, \\
X_s(z) &=  (1+(1-\tfrac{s}{\rho})t)^{\frac{1}{2} s^2-s} (1+(2-\tfrac{s}{\rho})t)^{-\frac{1}{2}s^2+\frac{1}{2}} (1+(1-\tfrac{s}{\rho})(2-\tfrac{s}{\rho})t)^{-\frac{1}{2}} (1+(1-\tfrac{s}{\rho})t)^{-\frac{(\rho-1)^2}{2\rho} s}, \\
Q_s(z) &= \tfrac{1}{2}t(1+(1-\tfrac{s}{\rho}) t), \quad R_s(z) = t, \quad T_s(z) = \rho t (1+ \tfrac{1}{2}(1-\tfrac{s}{\rho})(2-\tfrac{s}{\rho})t).
\end{align*}
Furthermore, $Y_{J,s}$, $Z_{J,s}$, $S_{J,s}$ are all algebraic functions.
\end{conjecture}

\begin{remark} \label{rk1mod} 
For $\rho=1$, we drop the assumption ``minimal of general type'' and we drop the subscript $J = \varnothing$ from the universal functions. In this case, the assumption $b_1(S) = 0$ is not needed when we require the sheaves of $M_S^H(1,c_1,c_2)$ to have fixed determinant so $M_S^H(1,c_1,c_2) \cong S^{[c_2]}$. 
In Section \ref{sec:generalconj}, we formulate a more general conjecture for \emph{any} smooth projective surface $S$ satisfying $p_g(S)>0$ and $b_1(S) = 0$. In this case, the formula is slightly more complicated and involves the Seiberg-Witten invariants of $S$.
\end{remark}

For $\rho=1$ and $L=u=0$, the first part of this conjecture follows from Theorem \ref{MOPthm1}.\footnote{In Conjecture \ref{conj1}, one can replace $\mu(L)$ by $x \, \mu(L)$, where $x$ is a formal variable. Then the corresponding virtual Segre invariants are given by the same formula with $L$ replaced by $x L$. When we say ``$L=u=0$'', we really mean ``$x=u=0$''.} In Theorem \ref{rk1prop}, we also prove formulae for $Q_{0}, S_{0}, T_{0}, Q_{1}, R_{1}, S_{1}, T_{1}$ by using the Hilbert-Chow morphism and Nakajima operators.

Surprisingly, the $\mu$-classes are related to the variables change $z=z(t)$. Indeed $z = z(t)$ \emph{equals} the inverse series of $R_s(z)$. We provide explicit formulae for $Y_{J,s}$, $Z_{J,s}$, $S_{J,s}$ for various values of $\rho,s$ in Section \ref{sec:alg}. In each case we get an algebraic expression.

For $\alpha=0$, the virtual Segre numbers reduce to (higher rank) Donaldson invariants
$$
\int_{[M_S^H(\rho,c_1,c_2)]^{\vir}} e^{\mu(L)  + \mu(\pt) u}.
$$
Then Conjecture \ref{conj1} (or rather its more general analog Conjecture \ref{conj1strong}) reduces to an algebraic version of the Mari\~{n}o-Moore conjecture for $\mathrm{SU}(\rho)$ Donaldson invariants \cite[(9.17)]{MM}, \cite[(10.107)]{LM}. The original Witten conjecture is an explicit formula for the $\mathrm{SU}(2)$ Donaldson invariants in terms of Seiberg-Witten invariants and was proved in the algebraic setting by the first-named author, H.~Nakajima, and K.~Yoshioka \cite{GNY3}. For ranks $\rho=3,4$, we determine the algebraic functions $Z_{J,s}$, $S_{J,s}$ allowing us to formulate explicit $\mathrm{SU}(3)$, $\mathrm{SU}(4)$ Witten conjectures (Section \ref{sec:Donaldson}). 

\subsection*{Virtual Verlinde numbers}

In order to generalize the line bundles $\mu(L) \otimes E^{\otimes r}$ on $S^{[n]}$ to any moduli space $M:=M_S^H(\rho,c_1,c_2)$, 
we recall the following construction from \cite[Ch.~8]{HL} (see also \cite{GNY2}). Recall that we assume $M$ has no strictly semistable sheaves and there exists a universal sheaf $\E$ on $S \times M$. We remove the latter assumption in Remark \ref{univdrop2}. Consider the map 
\begin{equation} \label{lambdaE}
\lambda_{\EE} : K(S) \rightarrow \Pic(M), \quad \alpha \mapsto \det \big( \pi_{M!} \big( \pi_S^* \alpha \cdot [\E] \big) \big)^{-1}.
\end{equation}
Let $c \in K(S)_{\mathrm{num}}$ be a class in the numerical Grothendieck group of $S$ such that $\rk(c) = \rho$, $c_1(c) = c_1$, and $c_2(c) = c_2$. Upon restriction to
\begin{equation*} 
K_c := \{ v \in K(S) \, : \, \chi(S,c \otimes v) = 0\},
\end{equation*}
the map $\lambda_{\EE} =: \lambda$ is independent of the choice of universal family $\EE$.

Fix $r \in \Z$, $L \in \Pic(S) \otimes \Q$ such that $\mathcal{L}:=L \otimes \det(c)^{-\frac{r}{\rho}} \in \Pic(S)$ and  $\rho$ divides $\mathcal{L}c_1 + r \big( \frac{1}{2} c_1(c_1-K) - c_2 \big)$. Take $v \in K(S)$ such that
\begin{itemize}
\item $\rk(v) = r$ and $c_1(v) = \mathcal{L}$, 
\item $c_2(v) = \frac{1}{2} \mathcal{L}(\mathcal{L}-K) +r\chi + \frac{1}{\rho} \mathcal{L} c_1 + \frac{r}{\rho}\big(\frac{1}{2}c_1(c_1-K) - c_2\big)$.
\end{itemize}
The second condition is equivalent to $v \in K_c \subset K(S)$. We define
\begin{equation} \label{muLEr}
\mu(L) \otimes E^{\otimes r} := \lambda(v).
\end{equation}
When $\rho = 1$ and $c_1=0$, this definition coincides with our previous definition of $\mu(L) \otimes E^{\otimes r}$ on $M_S^H(1,0,n) \cong S^{[n]}$ by \cite[Rem.~5.3(2)]{Got1}. For $r=0$, this recovers the definition of the Donaldson line bundle $\mu(L)$ studied in \cite{GNY2, GKW} by \cite[Rem.~5.3(1)]{Got1}.

Denote by $\O_M^{\vir}$ the virtual structure sheaf of $M$. We consider the virtual holomorphic Euler characteristics 
$$
\chi^{\vir}(M, \mu(L) \otimes E^{\otimes r}) := \chi(M, \mu(L) \otimes E^{\otimes r} \otimes \O_{M}^{\vir}).
$$
\begin{conjecture} \label{conj2}
Let $\rho \in \Z_{>0}$ and $r \in \Z$. There exist $G_r$, $F_{r} \in \C[[w]]$, $A_{J,r}$, $B_{J,r} \in \C[[w^{\frac{1}{2}}]]$ for all $J \subset [\rho-1]$ with the following property.\footnote{These universal functions depend on $\rho$ and $r$. We suppress the dependence on $\rho$.} Let $S$ be a minimal surface of general type with $b_1(S) = 0$ and $p_g(S)>0$, and let $L \in \Pic(S)$. Suppose $M:=M_S^H(\rho,c_1,c_2)$ contains no strictly semistable sheaves. Then the virtual Verlinde number $\chi^{\vir}(M, \mu(L) \otimes E^{\otimes r})$ equals the coefficient of $w^{\frac{1}{2}\vd(M)}$ of
\begin{align*}
\rho^{2 - \chi+K^2} \, G_{r}^{\chi(L)} \, F_{r}^{\frac{1}{2} \chi} \sum_{J \subset [\rho-1]} (-1)^{|J| \chi} \, \eps_{\rho}^{\|J\| K c_1} \, A_{J,r}^{K L} \, B_{J,r}^{K^2}.
\end{align*}
Furthermore, $A_{J,r}$, $B_{J,r}$ are all algebraic functions.
\end{conjecture}
For $\rho=1$, the first part of this conjecture is Theorem \ref{EGLthm} (and Remark \ref{rk1mod} applies). We present a stronger version of this conjecture, for any smooth projective surface with $p_g(S)>0$ and $b_1(S) = 0$, in Section \ref{sec:generalconj}.
If $S$ is a $K3$ surface, then $M$ is smooth of expected dimension. Moreover, $M$ is holomorphic symplectic and deformation equivalent to $S^{[\frac{1}{2}\vd(M)]}$ by \cite{OG,Huy, Yos}. As shown in \cite{GNY2}, using a result of A.~Fujiki, the Verlinde numbers of $M$ can be expressed in terms of those of $S^{[\frac{1}{2}\vd(M)]}$. Combining with Theorem \ref{EGLthm}, one obtains
\begin{align}
\begin{split} \label{fg}
G_r(w) &= g_{r/\rho}(w) = 1+v, \\
F_r(w) &= f_{r/\rho}(w) = (1+v)^{\frac{r^2}{\rho^2}}(1+\tfrac{r^2}{\rho^2} v)^{-1},
\end{split}
\end{align}
where we applied the formal variable change
$w = v(1+v)^{r^2 / \rho^2-1}.$ 
Hence Conjecture \ref{conj2} is true for $K3$ surfaces, which determines $G_r$ and $F_r$. 

\subsection*{Virtual Segre-Verlinde correspondence}

Let $\rho \in \Z_{>0}$ and $s \in \Z$. The universal functions of Conjectures \ref{conj1} and \ref{conj2}, cf.~\eqref{fg}, satisfy
\begin{align*}
f_{r/\rho}(w) &= V_s(z)^{\frac{s}{\rho} (\rho^{\frac{1}{2}} - \rho^{-\frac{1}{2}})^2} \, W_s(z)^{-\frac{4s}{\rho}} \, X_s(z)^2, \\
g_{r/\rho}(w) &= V_s(z) \, W_s(z)^2,
\end{align*}
where $s = \rho + r$ and we use the following formal variable changes
\begin{equation} \label{varchange2}
w = v(1+v)^{\tfrac{r^2}{\rho^2}-1}, \quad z = t (1+(1-\tfrac{s}{\rho}) t)^{1-\tfrac{s}{\rho}}, \quad v = t( 1- \tfrac{r}{\rho} t )^{-1}.
\end{equation}

\begin{conjecture} \label{conj3} 
For any $\rho>0$, $r \in \Z$, $s=\rho+r$, and under the formal variable change \eqref{varchange2}, we have 
\begin{align*}
A_{J,r}(w^{\frac{1}{2}}) = W_{s}(z) \, Y_{J,s}(z^{\frac{1}{2}}), \quad  B_{J,r}(w^{\frac{1}{2}}) = Z_{J,s}(z^{\frac{1}{2}}), 
\end{align*}
for all $J \subset [\rho-1]$.\footnote{Recall that the series $A_{J,r}, B_{J,r}$ and $Y_{J,s}, Z_{J,s}$ depend on $w^{\frac{1}{2}}$ and $z^{\frac{1}{2}}$, whereas $W_{s}$ depends on $z$ (Conj.~\ref{conj1} and \ref{conj2}). For the former we use the coordinate transformation $w^{\frac{1}{2}} = v^{\frac{1}{2}}(1+v)^{\frac{1}{2}(r^2/\rho^2-1)}$ etc.}
\end{conjecture}

\subsection*{Universal function}

We express the virtual Segre and Verlinde numbers for surfaces $S$ satisfying $b_1(S) = 0$ and $p_g(S)>0$ in terms of (descendent) Donaldson invariants of $S$. Then we apply Mochizuki's formula, which expresses Donaldson invariants in terms of Seiberg-Witten invariants and intersection numbers on products of Hilbert schemes of points. This leads to a universal function which essentially determines all virtual Segre and Verlinde numbers (Theorems \ref{structhm1} and \ref{structhm2}). Calculating these intersection numbers, up to a certain number of points, allows us to verify Conjectures \ref{conj1}, \ref{conj2}, \ref{conj3} (and their more general analogs Conjectures \ref{conj1strong}, \ref{conj2strong}) up to a certain order in several examples for ranks $\rho \leq 4$. The precise list of verifications can be found in Section \ref{sec:verif}. 

We call a virtual Segre/Verlinde number \emph{canonical} when $H = K_S$, $c_1$ is a multiple of $K_S$ such that $\gcd(\rho, Hc_1)=1$, and $\alpha \in \Z[\O(K_S)] \subset K(S)$ and $L$ is a power of $\O(K_S)$. Let $e(S)$ and $\sigma(S)$ denote the topological Euler characteristic and signature of $S$. We then prove the following:
\begin{theorem} \label{thminv}
Canonical virtual Segre and Verlinde numbers of smooth projective surfaces $S$ satisfying $b_1(S) = 0$ and $K_S$ very ample only depend on $e(S)$ and $\sigma(S)$.
\end{theorem}

\subsection*{Algebraicity, Galois actions, $K3$, virtual Serre duality.} 

In addition to the above-mentioned conjectures and results, this paper includes the following:
\begin{itemize}
\item We present several conjectural algebraic expressions for the remaining universal power series for ranks $\rho \leq 4$. These expressions have coefficients in Galois extensions of $\Q$ and are permuted under actions of corresponding Galois groups.
\item We present a new conjecture for Segre numbers of $K3$ surfaces (Conjecture \ref{conjK3}), which implies Conjecture \ref{conj1} for $S=K3$ and $L = u = 0$.\footnote{G.~Oberdieck \cite{Obe} recently proved Conj.~\ref{conjK3} using E.~Markman's work on monodromy operators.}
\item We conjecture and test relations among the universal power series suggested by virtual Serre duality (Conjecture \ref{conj:SDvir}). 
\item We discuss applications of our conjectures to surfaces containing a canonical curve with reduced connected components and we provide a blow-up formula.\footnote{The first-named author recently conjectured (and tested) blow-up formulae for the virtual Segre and Verlinde numbers of this paper \cite{Got2}. This provides further evidence for the conjectures in Sect.~\ref{sec:alg}.} 
\end{itemize} 

\noindent \textbf{Acknowledgements.} We thank Alina Marian, Dragos Oprea, and Rahul Pandharipande for useful discussions. We are grateful to Dragos Oprea for suggesting to specialize our conjectures to rigid sheaves as in \cite{MOP3}. Although we were unable to adapt the methods of loc.~cit.~to our setting, his suggestion led to the discovery of Conjecture \ref{conjK3} for $K3$ surfaces. M.K.~is supported by NWO grant VI.Vidi.192.012.

\section{Universal function} 
\addtocontents{toc}{\protect\setcounter{tocdepth}{2}}

\subsection{Mochizuki's formula} \label{sec:moc}

Let $(S,H)$ be a smooth polarized surface with $b_1(S) = 0$. Let $\rho \in \Z_{>0}$, $c_1 \in H^2(S,\Z)$ algebraic, and $c_2 \in H^2(S,\Z)$. We denote by $M:=M_S^H(\rho,c_1,c_2)$ the moduli space of rank $\rho$ Gieseker $H$-semistable sheaves on $S$ with Chern classes $c_1,c_2$. We assume $M$ does not contain strictly semistable sheaves. We also assume there exists a universal sheaf $\EE$ on $S \times M$, though we show in Remarks \ref{univdrop2} and \ref{univdrop1} how to drop this assumption. Then $M$ has a natural perfect obstruction theory with virtual tangent bundle  \cite{Moc}
$$
T_M^{\vir} = R\hom_{\pi_M}(\EE,\EE)_0[1],
$$ 
where $\pi_M : S \times M \rightarrow M$ is projection, $R\hom_{\pi_M} := R\pi_{M*} \circ R\hom$, and $(\cdot)_0$ denotes trace-free part. The resulting virtual class $[M]^{\vir}$ has virtual dimension given by \eqref{vd}. 

For any $\sigma \in H^*(S,\Q)$ and $k \geq 2$, we consider the slant product
$$
\ch_k(\EE) / \mathrm{PD}(\sigma) \in H^*(M,\Q).
$$
Suppose $P(\EE)$ is any polynomial expression in slant products (for various choices of $\sigma$ and $k$). Then we refer to
$$
\int_{[M]^{\vir}} P(\EE) \in \Q
$$
as a (descendent) Donaldson invariant of $S$. For any rank $\rho \in \Z_{>1}$, Mochizuki derived a remarkable expression for the Donaldson invariants of smooth projective surfaces $S$ with holomorphic 2-form \cite[Thm.~7.5.2]{Moc}. We will now present his formula.

When $p_g(S)>0$, we denote the Seiberg-Witten invariant of $S$ in class $a \in H^2(S,\Z)$ by $\SW(a)$. We follow Mochizuki's convention 
$\SW(a) = \widetilde{\SW}(2a-K_S)$, where $\widetilde{\SW}(b)$ is the usual Seiberg-Witten invariant for class $b \in H^2(S,\Z)$ from differential geometry. There are finitely many $a \in H^2(S,\Z)$ such that $\SW(a) \neq 0$ and such classes are called Seiberg-Witten basic classes. Seiberg-Witten basic classes $a$ satisfy $a^2 = a K_S$ \cite[Prop.~6.3.1]{Moc}; i.e.~the virtual dimension of the linear system $|a|$ is zero.

For any non-negative integers $\bsn = (n_1, \ldots, n_\rho)$, we define
$$
S^{[\bsn]} := S^{[n_1]} \times \cdots \times S^{[n_\rho]}.
$$
For a line bundle $L$ on $S$, we denote by $L^{[n_i]} := p_* q^* L$ the corresponding tautological bundle on $S^{[n_i]}$, where $q : \cZ_i \rightarrow S$ and $p : \cZ_i \rightarrow S^{[n_i]}$ are the projections from the universal subscheme $\cZ_i \subset S \times S^{[n_i]}$. We denote the pull-back of $L^{[n_i]}$ to $S^{[\bsn]}$ along projection by the same symbol. Writing $\I_i$ for the universal ideal sheaf on $S \times S^{[n_i]}$, we denote its pull-back to $S \times S^{[\bsn]}$ by the same symbol as well. Furthermore, we denote its twist by the pull-back of a divisor class $a_i$ on $S$ by $\I_i(a_i)$.

We consider the trivial torus action of $T = (\C^{*})^{\rho-1}$ on $S^{[\bsn]}$. We denote by 
$$
\t_1, \ldots, \t_{\rho-1} \in X(T) \cong \Z^{\rho-1}
$$
the degree one characters generating the character group $X(T)$. Then any character of $T$ is of the form $\bigotimes_i \t_i^{w_i}$ for some $w_1, \dots, w_{\rho-1} \in \Z$. Any $T$-equivariant coherent sheaf $\F$ on $S^{[\bsn]}$ decomposes into eigensheaves
$$
\F = \bigoplus_{\boldsymbol{w} = (w_1, \ldots, w_{\rho-1})} \F_{\boldsymbol{w}} \otimes \bigotimes_i \t_i^{w_i}.
$$
We also equip $S \times S^{[\bsn]}$ with the trivial $T$-action and we regard projection $\pi : S \times S^{[\bsn]} \rightarrow S^{[\bsn]}$ as a $T$-equivariant morphism. Furthermore, we write
$$
H_T^*(\mathrm{pt},\Z) = \Z[t_1^{\pm 1}, \ldots, t_{\rho-1}^{\pm 1}],
$$
where $t_i := c_1^T(\t_i)$ denotes the $T$-equivariant first Chern class of $\t_i$. The following (rational) characters in $X(T) \otimes \Q$ play an important role in Mochizuki's formula
\begin{align}
\begin{split} \label{defT}
\frakT_i := \t_i^{-1} \otimes \bigotimes_{j<i} \t_j^{\frac{1}{\rho-j}}, \quad \frakT_\rho := \bigotimes_{j<\rho} \t_j^{\frac{1}{\rho-j}}, \quad T_j := c_1^{T}(\frak{T}_j),
\end{split}
\end{align}
for all $i=1, \ldots, \rho-1$ and $j=1, \ldots, \rho$.

For any $\ch \in H^*(S,\Q)$, we define
$$
\chi(\ch) := \int_S \ch \cdot \td(S)
$$
and we denote the corresponding Hilbert polynomial by $h_{\ch}(m) = \chi(\ch \cdot e^{mH})$.
Furthermore, for any $a \in H^2(S,\Z)$, we define $\chi(a) := \chi(e^a)$.  

Suppose $P(\E)$ is any polynomial expression in slant products such that $$P(\E) = P(\E \otimes \mathcal{L})$$ for any $\mathcal{L} \in \Pic(S \times M)$. Then $P(\E)$ is independent of the choice of universal sheaf. For any $T$-equivariant coherent sheaf $\F$ on $S \times S^{[\bsn]}$, we denote by $P(\F)$ the expression obtained by replacing $S \times M$ by $S \times S^{[\bsn]}$, $\E$ by $\F$, and all Chern characters by $T$-equivariant Chern characters. Furthermore, for any divisor classes $\bsa = (a_1, \ldots, a_{\rho})$, we define
\begin{align*}
&Q\Big( \I_1(a_1) \otimes \frak{T}_1, \ldots, \I_\rho(a_{\rho}) \otimes \frak{T}_{\rho} \Big) := \\
&\prod_{i<j} e\Big( - R\hom_\pi(\I_i(a_i) \otimes \frakT_i , \I_j(a_j) \otimes \frakT_j) - R\hom_\pi(\I_j(a_j) \otimes \frakT_j , \I_i(a_i) \otimes \frakT_i) \Big),
\end{align*}
where $e$ denotes $T$-equivariant Euler class and
$$
\pi : S \times S^{[\bsn]} \rightarrow S^{[\bsn]}
$$
is the projection. Following Mochizuki \cite[Sect.~7.5.2]{Moc}, for any non-negative integers $\bsn = (n_1, \ldots, n_\rho)$ and any divisor classes $\bsa = (a_1, \ldots, a_{\rho})$ on $S$, we define
\begin{align*}
&\widetilde{\Psi}(\bsa,\bsn,\boldsymbol{t}) := \Bigg( \prod_{i=1}^{\rho-1} t_i^{\sum_{j \geq i} \chi(1,a_j,\frac{1}{2}a_j^2 - n_j)} \Bigg) \Bigg( \prod_{i<j} \frac{1}{(T_j - T_i)^{\chi(a_j)}} \Bigg) \\
&\cdot \frac{P\Big( \bigoplus_{i=1}^{\rho} \I_i(a_i) \otimes \frak{T}_i \Big)}{Q\Big( \I_1(a_1) \otimes \frak{T}_1, \ldots, \I_\rho(a_\rho) \otimes \frak{T}_\rho \Big)} \Bigg( \prod_{i=1}^{\rho-1} e(\O(a_i)^{[n_i]}) \Bigg) \Bigg( \prod_{i<j} e(\O(a_j)^{[n_j]} \otimes \frakT_j \frakT_i^{-1}) \Bigg).
\end{align*}
Finally, we define
\begin{align*}
\Psi(\bsa,\bsn) := \mathrm{Coeff}_{t_{1}^0} \cdots \mathrm{Coeff}_{t_{r-1}^0} \widetilde{\Psi}(\bsa,\bsn,\boldsymbol{t}),
\end{align*}
where $\mathrm{Coeff}_{t_i^0}(\cdot)$ takes the coefficient of $t_{i}^{0}$ after expansion of $(\cdot)$ as a Laurent series in $t_i$. Recall from the introduction that we write $K:=K_S$, $\chi:=\chi(\O_S)$.

\begin{theorem}[Mochizuki] \label{mocthm}
Let $(S,H)$ be a smooth polarized surface such that $b_1(S) = 0$ and $p_g(S) > 0$. Let $\rho \in \Z_{>1}$, $c_1 \in H^2(S,\Z)$ algebraic, $c_2 \in H^4(S,\Z)$, and consider $M:=M_S^H(\rho,c_1,c_2)$. Assume the following:
\begin{enumerate}
\item $M$ does not contain strictly semistable sheaves,
\item there exists a universal sheaf $\E$ on $S \times M$,
\item $h_{(\rho, c_1, \frac{1}{2} c_1^2-c_2)} / \rho > h_{e^{K}}$,
\item $\chi(\rho, c_1, \frac{1}{2} c_1^2-c_2) > (\rho-2) \chi$.
\end{enumerate}
Let $P(\EE)$ be a polynomial expression in slant products such that $P(\E) = P(\E \otimes \mathcal{L})$ for any $\mathcal{L} \in \Pic(S \times M)$. Then
\begin{align*}
\int_{[M]^{\vir}} P(\EE) = (-1)^{\rho-1} \rho \sum_{(a_1, \ldots, a_{\rho}) \atop (n_1, \cdots, n_\rho)} \prod_{i=1}^{\rho-1} \SW(a_i) \int_{S^{[\boldsymbol{n}]}} \Psi(\boldsymbol{a}, \boldsymbol{n}),
\end{align*}
where the sum is over all $(a_1, \ldots, a_{\rho}) \in H^2(S,\Z)^\rho$ and $(n_1, \ldots, n_\rho) \in \Z_{\geq 0}^\rho$ satisfying
\begin{align}
\begin{split} \label{Mochineq}
c_1 &= a_1 + \cdots + a_{\rho} \\
c_2 &= n_1 + \cdots + n_\rho + \sum_{1 \leq i<j \leq \rho } a_i a_j \\
h_{(1,a_i,\frac{1}{2}a_i^2 - n_i)} &< \frac{1}{\rho - i} \sum_{j>i} h_{(1,a_j,\frac{1}{2}a_j^2 - n_j)} \quad \forall i=1, \ldots, \rho-1.
\end{split}
\end{align}
\end{theorem}

\subsection{Main theorems} \label{sec:mainthms}

From Mochizuki's formula, we derive two theorems, which give universal functions for the virtual Segre and Verlinde numbers of smooth projective surfaces satisfying $b_1(S)=0$ and $p_g(S)>0$. 
\begin{theorem} \label{structhm1}
Let $\rho \in \Z_{>1}$ and $s \in \Z$. There exist 
$A_{\underline{c_1(\alpha)^2},s}$, $A_{\underline{c_1(\alpha)L},s}$, $A_{\underline{c_1(\alpha)K},s}$, $A_{\underline{c_2(\alpha)},s}$, $A_{\underline{L^2},s}$, $A_{\underline{LK},s}$, $A_{\underline{\pt},s}$, $A_{\underline{K^2},s}$, $A_{\underline{\chi},s}$, $A_{\underline{a_ic_1(\alpha)},s}$, $A_{\underline{a_iL},s}$, $A_{\underline{a_iK},s}$, $A_{\underline{a_ia_j},s} \in 1+ q\Q(\!(t_1,\ldots, t_{\rho-1})\!)[[q]]$ for all $1 \leq i \leq j \leq \rho$ with the following property.
Let $(S,H)$ be a smooth polarized surface such that $b_1(S) = 0$ and $p_g(S) > 0$. Let $c_1 \in H^2(S,\Z)$ algebraic, $c_2 \in H^4(S,\Z)$, and consider $M:=M_S^H(\rho,c_1,c_2)$. Let $\alpha \in K(S)$ with $\rk \alpha = s$ and $L \in \Pic(S)$. Assume the following:
\begin{enumerate}
\item[(i)] $M$ does not contain strictly semistable sheaves,
\item[(ii)] there exists a universal sheaf $\E$ on $S \times M$,
\item[(iii)] $h_{(\rho, c_1, \frac{1}{2} c_1^2-c_2)} / \rho > h_{e^{K}}$,
\item[(iv)] $\chi(\rho, c_1, \frac{1}{2} c_1^2-c_2) > (\rho-2) \chi$,
\item[(v)] for any $a_1, \ldots, a_\rho \in H^2(S,\Z)$ satisfying $a_1 + \cdots +a_\rho=c_1$, such that $a_1, \ldots, a_{\rho-1}$ are Seiberg-Witten basic classes and $a_i H \leq \frac{1}{\rho-i} \sum_{j>i} a_jH$ for all $i=1, \ldots, \rho-1$, these inequalities are strict.\footnote{Here $a_iH$ denotes the intersection number of the elements $a_i, H \in H^2(S,\Z)$.}
\end{enumerate}
Then $\int_{[M]^{\vir}} c(\alpha_M) \, \exp(\mu(L) + \mu(\pt)u)$ equals $\mathrm{Coeff}_{t_{1}^0} \cdots \mathrm{Coeff}_{t_{r-1}^0}$ of the coefficient of $z^{\frac{1}{2} \vd(M)}$ of the following expression
\begin{align*}
&(-1)^{\rho-1} \rho \sum_{(a_1, \ldots, a_{\rho})} \prod_{i=1}^{\rho-1} \SW(a_i) t_i^{\sum_{j \geq i} ( \frac{1}{2} a_j(a_j-K) + \chi  ) }  \prod_{i=1}^{\rho} (1+T_i)^{-\chi(\alpha \otimes \O_S(a_i - \frac{c_1}{\rho}))} e^{- (a_iL) \, T_i -\frac{1}{2} T_i^2 u} \\
&\cdot  \prod_{1 \leq i<j \leq \rho}  (T_j - T_i)^{\chi(a_j-a_i) - \chi(a_j)} (T_i-T_j)^{\chi(a_i-a_j)} A_{\underline{c_1(\alpha)^2},s}^{c_1(\alpha)^2}  A_{\underline{c_1(\alpha)L},s}^{c_1(\alpha)L}  A_{\underline{c_1(\alpha)K},s}^{c_1(\alpha)K} A_{\underline{c_2(\alpha)},s}^{c_2(\alpha)}  A_{\underline{L^2},s}^{L^2}  A_{\underline{LK},s}^{LK} A_{\underline{\pt},s}^u \\
&\cdot A_{\underline{K^2},s}^{K^2} \cdot (z^{-\frac{\rho^2- 1}{2}} A_{\underline{\chi},s})^{\chi} \prod_{i=1}^{\rho} A_{\underline{a_ic_1(\alpha)},s}^{a_ic_1(\alpha)} A_{\underline{a_iL},s}^{a_iL} A_{\underline{a_iK},s}^{a_iK} \cdot (z^{-\frac{\rho-1}{2}} A_{\underline{a_ia_i},s})^{a_i^2} \prod_{1 \leq i<j \leq \rho} ( zA_{\underline{a_ia_j},s})^{a_ia_j},
\end{align*}
where $T_i$ was defined in \eqref{defT}, the sum is over all $(a_1, \ldots, a_\rho) \in H^2(S,\Z)^\rho$ satisfying $a_1 + \cdots +a_\rho=c_1$ and $a_i H \leq  \frac{1}{\rho-i} \sum_{j>i} a_jH$ for all $i=1, \ldots, \rho-1$, and all power series are evaluated at $q=z^\rho$.
\end{theorem}
\begin{proof}
The proof is divided into three steps. \\

\noindent \textbf{Step 1.} Applying Grothendieck-Riemann-Roch to the projection $\pi : S \times M \rightarrow M$, we obtain a polynomial expression $P(\E)$ in slant products such that 
$$
c(\alpha_M) \, e^{\mu(L) + \mu(\pt) u} = P(\EE).
$$

\noindent \textbf{Step 2.} By Step 1 and Theorem \ref{mocthm} 
$$
\int_{[M]^{\vir}} c(\alpha_M) \, e^{\mu(L) + \mu(\pt) u}
$$ 
equals $\mathrm{Coeff}_{t_{1}^0} \cdots \mathrm{Coeff}_{t_{r-1}^0}$ of 
$$
(-1)^{\rho-1} \rho \sum_{a_1 + \cdots + a_{\rho}=c_1 \atop a_i H \leq \frac{1}{\rho-i} \sum_{j>i} a_jH \ \forall i}  \sum_{n_1 + \cdots + n_\rho =c_2 - \sum_{i<j} a_ia_j} \prod_{i=1}^{\rho-1} \SW(a_i) \int_{S^{[\boldsymbol{n}]}} \widetilde{\Psi}(\boldsymbol{a}, \boldsymbol{n},\boldsymbol{t}),
$$
where we used assumption (v) in order to simplify  \eqref{Mochineq}. Note that $\widetilde{\Psi}$ also depends on $\alpha,L,u$ (though our choice of $P(\EE)$) but we suppress this dependence. Consider
$$
\sum_{(n_1, \ldots, n_\rho) \in \Z_{\geq 0}^n} q^{n_1+ \cdots +n_\rho} \int_{S^{[\boldsymbol{n}]}} \widetilde{\Psi}(\boldsymbol{a}, \boldsymbol{n}, \boldsymbol{t}).
$$
The perturbative term $\mathsf{Pert} := \widetilde{\Psi}(\boldsymbol{a},\boldsymbol{0}, \boldsymbol{t})$ is defined as its constant term, which arises by taking $n_1 = \cdots = n_\rho=0$. Explicitly, $\mathsf{Pert}$ equals
\begin{align*}
&\prod_{i=1}^{\rho-1} t_i^{\sum_{j \geq i} ( \frac{1}{2} a_j(a_j-K) + \chi  ) } \prod_{1 \leq i<j \leq \rho}  (T_j - T_i)^{\chi(a_j-a_i) - \chi(a_j)} (T_i-T_j)^{\chi(a_i-a_j)} \\
&\cdot c^{T}(-R\Gamma(S, \alpha \otimes \O_S(-c_1 / \rho) \otimes \bigoplus_{i=1}^{\rho} \O_S(a_i) \otimes \frakT_i )) \\
&\cdot \exp\Big(    \int_S (c_1(L) + u [\pt])  \Big\{c^T_2\Big(   \bigoplus_{i=1}^{\rho} \O_S(a_i) \otimes \frakT_i   \Big) - \frac{\rho-1}{2\rho} c^T_1\Big(   \bigoplus_{i=1}^{\rho} \O_S(a_i) \otimes \frakT_i   \Big)^2  \Big\} \Big),
\end{align*}
where $c^T$ denotes $T$-equivariant Chern class and $\int_S$ denotes $T$-equivariant integration. Using $\sum_i T_i = 0$, this can be simplified to
\begin{align*}
\mathsf{Pert} = &\prod_{i=1}^{\rho-1} t_i^{\sum_{j \geq i} ( \frac{1}{2} a_j(a_j-K) + \chi  ) } \prod_{1 \leq i<j \leq \rho}  (T_j - T_i)^{\chi(a_j-a_i) - \chi(a_j)} (T_i-T_j)^{\chi(a_i-a_j)} \\
&\cdot \prod_{i=1}^{\rho} (1+T_i)^{-\chi(\alpha \otimes \O_S(a_i - \frac{c_1}{\rho}))} e^{ - (a_iL) \, T_i -\frac{1}{2} T_i^2 u}.
\end{align*}

\noindent \textbf{Step 3.} Let $S$ be any \emph{possibly disconnected} smooth projective surface. Take any $\alpha \in K(S)$ and arbitrary divisor classes $L,a_1, \ldots, a_\rho$ on $S$. Consider the generating function
\begin{equation} \label{defZ}
\mathsf{Z}_S(\alpha,L,\boldsymbol{a},\boldsymbol{t},u,q) := \frac{1}{\mathsf{Pert}} \sum_{(n_1, \ldots, n_\rho) \in \Z_{\geq 0}^{\rho}} q^{n_1+ \cdots +n_\rho} \int_{S^{[\boldsymbol{n}]}} \widetilde{\Psi}(\boldsymbol{a}, \boldsymbol{n}, \boldsymbol{t}),
\end{equation}
which has constant term equal to 1. For any $(S',\alpha',L',\boldsymbol{a}')$ and $(S'',\alpha'',L'',\boldsymbol{a}'')$ we have
\begin{align}
\begin{split} \label{mult}
\mathsf{Z}_{S' \sqcup S''}(\alpha' \oplus \alpha'',L'\oplus L'',\boldsymbol{a}' \oplus \boldsymbol{a}'' ,\boldsymbol{t},u,q) = \mathsf{Z}_{S'}(\alpha',L',\boldsymbol{a}' ,\boldsymbol{t},u,q) \mathsf{Z}_{S''}(\alpha'',L'',\boldsymbol{a}'' ,\boldsymbol{t},u,q).
\end{split}
\end{align}
Consider the decomposition
\begin{align*}
&S^{[n_1]} \times \cdots \times S^{[n_\rho]} = \bigsqcup_{n_1'+n_1''=n_1} \cdots \bigsqcup_{n_\rho'+n_\rho'' = n_\rho} S^{\prime [n_1']} \times  \cdots \times S^{\prime [n_\rho']} \times S^{\prime \prime [n_1'']} \times \cdots \times S^{\prime\prime [n_\rho'']} \\
&\bigoplus_{i=1}^{\rho} \I_i(a_i) \otimes \frakT_i\Big|_{S \times S^{\prime [\boldsymbol{n}']} \times S^{\prime \prime [\boldsymbol{n}'']}} = \bigoplus_{i=1}^{\rho} j'_*(\I'_i(a'_i) \otimes \frakT_i) \oplus j''_* (\I''_i(a''_i) \otimes \frakT_i),
\end{align*}
where $j' : S' \times S^{\prime [\boldsymbol{n}']} \times S^{\prime \prime [\boldsymbol{n}'']} \hookrightarrow S \times S^{\prime [\boldsymbol{n}']} \times S^{\prime \prime [\boldsymbol{n}'']}$, $j'' : S'' \times S^{\prime [\boldsymbol{n}']} \times S^{\prime \prime [\boldsymbol{n}'']} \hookrightarrow S \times S^{\prime [\boldsymbol{n}']} \times S^{\prime \prime [\boldsymbol{n}'']}$ denote the inclusions, and we suppress various pull-backs. The multiplicative property \eqref{mult} follows from these decompositions combined with
\begin{align*}
c(E + F) &= c(E) c(F), \quad e(E + F) = e(E) e(F) \\
\exp(\pi_*(\beta \cdot \ch_2(E+F))) &= \exp(\pi_*(\beta \cdot \ch_2(E))) \exp(\pi_*(\beta \cdot \ch_2(F))).
\end{align*}

The multiplicative property \eqref{mult} combined with universality of intersection numbers on Hilbert schemes \cite[Thm.~4.1]{EGL} implies the existence of the universal functions $A_{\underline{c_1(\alpha)^2},s}, \ldots, A_{\underline{a_ia_j},s}$ depending \emph{only} on $\rho,s$ and satisfying\footnote{This part of the argument is well-known and has been used in many settings e.g.~\cite{GNY1, KST, GK1, GK2, GK3, Laa1, Laa2}. Also note that we use a slight enhancement of \cite[Thm.~4.1]{EGL} to intersection numbers on \emph{products} of Hilbert schemes, which was first treated in \cite{GNY1}.}
\begin{align}
\begin{split} \label{ZA}
\mathsf{Z}_S(\alpha,L,\boldsymbol{a} ,\boldsymbol{t},u,q) =  &A_{\underline{c_1(\alpha)^2},s}^{c_1(\alpha)^2}  A_{\underline{c_1(\alpha)L},s}^{c_1(\alpha)L}  A_{\underline{c_1(\alpha)K},s}^{c_1(\alpha)K} A_{\underline{c_2(\alpha)},s}^{c_2(\alpha)}  A_{\underline{L^2},s}^{L^2}  A_{\underline{LK},s}^{LK} A_{\underline{\pt},s}^u   A_{\underline{K^2},s}^{K^2} A_{\underline{\chi},s}^{\chi}  \\
&\prod_{i=1}^{\rho} A_{\underline{a_ic_1(\alpha)},s}^{a_ic_1(\alpha)} A_{\underline{a_iL},s}^{a_iL} A_{\underline{a_iK},s}^{a_iK}  \prod_{1 \leq i \leq j \leq \rho}  A_{\underline{a_ia_j},s}^{a_ia_j},
\end{split}
\end{align}
for any $(S,\alpha,L,\boldsymbol{a})$. 
Finally, we note that
\begin{align*}
\vd &= 2 \rho c_2 - (\rho-1) c_1^2 - (\rho^2-1) \chi \\
&= 2\rho \sum_i n_i + 2 \sum_{i<j} a_i a_j - (\rho-1) \sum_i a_i^2 - (\rho^2-1)\chi.
\end{align*}
Referring to equation \eqref{defZ}, this implies that we obtain the desired result by making the substitution $q = z^\rho$
\end{proof}

\begin{theorem} \label{structhm2}
Let $\rho \in \Z_{>1}$ and $r \in \Z$. There exist $B_{\underline{L^2},r}$, $B_{\underline{LK},r}$, $B_{\underline{K^2},r}$, $B_{\underline{\chi},r}$, $B_{\underline{a_iL},r}$, $B_{\underline{a_iK},r}$, $B_{\underline{a_ia_j},r} \in 1+ q\Q(\!(t_1,\ldots, t_{\rho-1})\!)[[q]]$, for all $1 \leq i \leq j \leq \rho$, with the following property.
Let $(S,H)$ be a smooth polarized surface such that $b_1(S) = 0$, $p_g(S) > 0$, and let $L \in \Pic(S)$. Let $c_1 \in H^2(S,\Z)$ algebraic, $c_2 \in H^4(S,\Z)$, and consider $M:=M_S^H(\rho,c_1,c_2)$. 
Assume the hypotheses (i)--(v) of Theorem \ref{structhm1} hold.
Then $\chi^{\vir}(M,\mu(L) \otimes E^{\otimes r})$ equals $\mathrm{Coeff}_{t_{1}^0} \cdots \mathrm{Coeff}_{t_{r-1}^0}$ of the coefficient of $w^{\frac{1}{2} \vd(M)}$ of the following expression
\begin{align*}
&(-1)^{\rho-1} \rho \sum_{(a_1, \ldots, a_{\rho})} \prod_{i=1}^{\rho-1} \SW(a_i)  t_i^{\sum_{j \geq i} ( \frac{1}{2} a_j(a_j-K) + \chi ) } \prod_{i=1}^{\rho} e^{-(\frac{r}{2} a_i^2 + (L - \frac{r}{\rho}c_1) a_i) T_i} \\
&\cdot \prod_{1 \leq i<j \leq \rho}  \frac{(1-e^{-(T_j - T_i)})^{\chi(a_j - a_i)} (1-e^{-(T_i - T_j)})^{\chi(a_i - a_j)}}{(T_j - T_i)^{\chi(a_j)}} \\
&\cdot B_{\underline{L^2},r}^{L^2} B_{\underline{LK},r}^{LK}  B_{\underline{K^2},r}^{K^2}  \cdot (w^{-\frac{\rho^2 - 1}{2}}B_{\underline{\chi},r})^{\chi} \prod_{i=1}^{\rho}  B_{\underline{a_iL},r}^{a_iL} B_{\underline{a_iK},r}^{a_iK} \cdot (w^{-\frac{\rho-1}{2}} B_{\underline{a_ia_i},r})^{a_i^2} \prod_{1 \leq i<j \leq \rho} (w B_{\underline{a_ia_j},r})^{a_ia_j},
\end{align*}
where the sum is over all $(a_1, \ldots, a_\rho) \in H^2(S,\Z)^\rho$ satisfying $a_1 + \cdots +a_\rho=c_1$ and $a_i H \leq \frac{1}{\rho-i} \sum_{j>i} a_jH$ for all $i=1, \ldots, \rho-1$, and all power series are evaluated at $q=w^\rho$.
\end{theorem}
\begin{proof}
The proof is similar to that of Theorem \ref{structhm1}. We indicate the differences. \\

\noindent \textbf{Step 1:} By the virtual Hirzebruch-Riemann-Roch theorem \cite{FG}, we have
$$
\chi^{\vir}(M,\mu(L) \otimes E^{\otimes r}) = \int_{[M]^{\vir}} e^{c_1(\lambda_{\E}(v))} \, \td(T_{M}^{\vir}),
$$
where $v$ and $\lambda_{\E}(v) = \det(\pi_{M!} ( \pi_S^* v \cdot [\E]))^{-1}$ were defined in the introduction. The virtual tangent bundle is given by
$$
T^{\vir}_{M} = R \hom_\pi(\E,\E)_0[1],
$$ 
where $\pi : S \times M \rightarrow M$ denotes projection. Applying Grothendieck-Riemann-Roch gives
$$
\ch(T_{M}^{\vir}) = \pi_{M*} \big( ( 1 - \ch(\E)^{\vee} \cdot  \ch(\E)) \cdot \pi_{S}^* \td(T_S) \big).
$$
Hence $\td(T_{M}^{\vir})$ can be written as a polynomial in expressions of the form
\begin{equation} \label{prodch}
\pi_{M*} \big( \ch_{i}(\EE) \ch_{j}(\EE) \cdot \pi_S^* \beta \big),
\end{equation}
for certain classes $\beta$.  Expression \eqref{prodch} can be written as a polynomial expression in slant products as follows. Consider $S \times S \times M$ and denote by $\pi_{i}$ projection onto the $i$th component and by $\pi_{ij}$ projection onto components $i$ and $j$. Then \eqref{prodch} equals
$$
\pi_{3*} \big( \pi_1^* \beta \cdot \pi_{13}^* \ch_{i}(\E) \cdot \pi_{23}^* \ch_{j}(\E) \cdot \pi_{12}^* \mathrm{PD}(\Delta) \big),
$$
where $\Delta$ is the diagonal inside $S \times S$. Consider the K\"unneth decomposition
$$
\mathrm{PD}(\Delta) = \sum_{a+b = 4} \theta_1^{(a)} \boxtimes \theta_2^{(b)},
$$
where $\theta_i^{(a)} \in H^{a}(S,\Q)$. Then we can write \eqref{prodch} as 
$$
\sum_{a+b= 4} \big( \ch_{i}(\EE) / (\beta \cdot \theta_1^{(a)})  \big)  \cdot \big( \ch_{j}(\EE) / \theta_2^{(b)}  \big).
$$

Since $c_1(\det E) = c_1(E)$, for any complex of sheaves $E$, we find 
$$
c_1(\lambda(v)) = - c_1( \pi_{M!} ( \pi_S^* v \cdot [\E]) ),
$$
which, using Grothendieck-Riemann-Roch, can also be expressed in slant products. \\

\noindent \textbf{Step 2:} The perturbative term $\mathsf{Pert} = \widetilde{\Psi}(\boldsymbol{a}, \boldsymbol{0},\boldsymbol{t})$ is given by
\begin{align*}
&\prod_{i=1}^{\rho-1} t_i^{\sum_{j \geq i} ( \frac{1}{2} a_j(a_j-K) + \chi  ) } \prod_{1 \leq i<j \leq \rho}  (T_j - T_i)^{\chi(a_j-a_i) - \chi(a_j)} (T_i-T_j)^{\chi(a_i-a_j)} \\
&\cdot e^{- c_1^T(\det \pi_{!}  (v \cdot \sum_{i=1}^{\rho} \O_S(a_i) \otimes \frakT_i  ))} \prod_{1 \leq i<j \leq \rho} \td^T(-R\Gamma(S,\O_S(a_j-a_i)) \otimes \frakT_j \frakT_i^{-1} -R\Gamma(S,\O_S(a_i-a_j)) \otimes \frakT_i \frakT_j^{-1}),
\end{align*}
where $c_1^T$ denotes $T$-equivariant first Chern class and $\pi : K^T(S) \rightarrow K^T(\pt)$ is $T$-equivariant push-forward to a point. Using $\sum_i T_i = 0$, this simplifies to
\begin{align*}
&\prod_{i=1}^{\rho-1} t_i^{\sum_{j \geq i} ( \frac{1}{2} a_j(a_j-K) + \chi  ) }  \prod_{1 \leq i<j \leq \rho}  (T_j - T_i)^{\chi(a_j-a_i) - \chi(a_j)} (T_i-T_j)^{\chi(a_i-a_j)} \\ 
&\cdot \prod_{i=1}^{\rho} e^{- (\frac{r}{2} a_i^2 + (L - \frac{r}{\rho}c_1) a_i) T_i} \prod_{1 \leq i<j \leq \rho} \Bigg( \frac{T_j - T_i}{1-e^{-(T_j - T_i)}} \Bigg)^{-\chi(a_j - a_i)} \Bigg( \frac{T_i - T_j}{1-e^{-(T_i - T_j)}} \Bigg)^{-\chi(a_i - a_j)}.
\end{align*}

\noindent \textbf{Step 3:} Analogous to Step 3 of Theorem \ref{structhm1}. This time, the multiplicative property \eqref{mult} requires the following identity
\begin{align*}
\td(E + F) &= \td(E) \td(F). \qedhere
\end{align*}
\end{proof}

\begin{remark} \label{univdrop2}
We now show that Condition (ii), i.e.~the existence of a global universal sheaf $\E$ on $S \times M$, can be dropped from Theorem \ref{structhm2}. Without the existence of a global universal sheaf, there still exists a homomorphism as in \cite[Sect.~8.1]{HL}, \cite[Sect.~1.1]{GNY2}
\begin{equation*} 
\lambda : K_c \longrightarrow \Pic(M) 
\end{equation*}
with the following property. Take a universal sheaf $\E$ on $S \times M'$, where $\phi : M' \rightarrow M$ is an \'etale cover (such a ``twisted'' universal sheaf always exists). Then
$$
\phi^* \lambda(v) = \lambda_{\E}(v)
$$ 
for all $v \in K_c$, where $\lambda_{\E}$ is defined as in \eqref{lambdaE} with $M$ replaced by $M'$. Therefore, we can define $\mu(L) \otimes E^{\otimes r}$ as in \eqref{muLEr} with $\lambda_{\E}$ replaced by $\lambda$. Hence the virtual Verlinde numbers are defined without assuming the existence of a global universal sheaf. The original version of Mochizuki's formula \cite[Thm.~7.5.2]{Moc} holds on the Deligne-Mumford stack $\mathcal{M}$ of oriented sheaves, i.e.~pairs $(F,\chi)$ where $[F] \in M$ and $\chi : F \cong \O(c_1)$. Consider the degree $\frac{1}{\rho} : 1$ \'etale morphism $f : \mathcal{M} \rightarrow M$. Then $f_* [\mathcal{M}]^{\vir} = \tfrac{1}{\rho} [M]^{\vir}$. There always exists a universal sheaf $\mathcal{E}$ on $S \times \mathcal{M}$ and we have
\begin{align*}
f^* T^{\vir}_{M} \cong T^{\vir}_{\mathcal{M}}, \quad f^* \lambda(v) = \lambda_{\mathcal{E}}(v),
\end{align*}
for all $v \in K_c$. Hence
$$
\chi^{\vir}(M,\mu(L) \otimes E^{\otimes r}) = \rho \cdot \int_{[\mathcal{M}]^{\vir}} e^{c_1(\lambda_{\mathcal{E}}(v))} \, \td(T_{\mathcal{M}}^{\vir})
$$
and one can simply repeat the proof of Theorem \ref{structhm2} using Mochizuki's formula on $\mathcal{M}$. This shows Condition (ii) can be dropped from Theorem \ref{structhm2}.
\end{remark}

\begin{remark} \label{univdrop1}
We now show that Condition (ii), i.e.~the existence of a global universal sheaf $\E$ on $S \times M$ can be dropped from Theorem \ref{structhm1} as well. Recall that $\ch(\alpha_M)$ was defined by a formal application of the Grothendieck-Riemann-Roch formula in the introduction
\begin{equation} \label{powerroot}
\ch(\alpha_M) :=   - \pi_{M*} \Big( \ch(\E) \cdot  \ch(\det(\E))^{-\frac{1}{\rho}} \cdot \pi_S^* \ch(\alpha) \cdot \pi_S^* \td(S) \Big) \in A^*(S)_{\Q}. 
\end{equation}
When $\E$ does not exist globally on $S \times M$, the sheaf $\E^{\otimes \rho} \otimes \det(\E)^{-1}$ still exists globally on $S \times M$ (essentially because this expression is invariant under replacing $\E$ by $\E \otimes \mathcal{L}$ so it glues from local \'etale patches). Hence $\ch(\E^{\otimes \rho} \otimes \det(\E)^{-1})^{1/\rho} \in A^*(S)_{\Q}$ is defined and we simply replace $\ch(\E) \cdot  \ch(\det(\E))^{-\frac{1}{\rho}}$ by this expression in \eqref{powerroot}. Obviously, when $\E$ exists globally on $S \times M$, we have

$$
\ch(\E^{\otimes \rho} \otimes \det(\E)^{-1})^{\frac{1}{\rho}} = \ch(\E) \cdot \ch(\det(\E))^{-\frac{1}{\rho}}
$$
and we recover the previous definition. Hence the virtual Segre numbers are defined without assuming the existence of a global universal sheaf.  Using the morphism $f : \mathcal{M} \to M$ from the previous remark and noting $f^*(\E^{\otimes \rho} \otimes \det(\E)^{-1}) \cong \mathcal{E}^{\otimes \rho} \otimes \det(\mathcal{E})^{-1}$, the virtual Segre numbers can be lifted to $\mathcal{M}$ and one can repeat the proof of Theorem \ref{structhm1} using Mochizuki's formula on $\mathcal{M}$. This shows Condition (ii) can be dropped from Theorem \ref{structhm1}.
\end{remark}

\begin{remark}
Conjecturally,  Conditions (iii) and (v) can also be dropped from Theorems \ref{structhm1}, \ref{structhm2} and the sum in the formula can be replaced by the sum over \emph{all} classes $(a_1, \ldots, a_\rho) \in H^2(S,\Z)^\rho$ satisfying $a_1 + \cdots +a_\rho=c_1$. See also \cite{GNY1, GK1, GK2, GK3, GKW}. In some of the calculations described in Sections \ref{sec:verif:higherrk} and \ref{sec:alg}, we use this strong form of Mochizuki's formula.
\end{remark}

\begin{remark}
Condition (iv) is essential and can be equivalently stated as 
$$
c_2 < \tfrac{1}{2}c_1(c_1-K) + 2\chi.
$$ 
At first glance, this appears a very strong restriction. However, applying $- \otimes \O_S(\ell H)$ induces an isomorphism of moduli spaces of Gieseker $H$-stable sheaves
$$
M_{S}^{H}(\rho,c_1,c_2)\cong M_S^H(\rho,c_1+ \rho \ell H,c_2+(\rho-1)\ell Hc_1+\tfrac{1}{2}\rho(\rho-1)\ell^2 H^2)
$$ 
leaving corresponding virtual Segre/Verlinde numbers unchanged. Under this isomorphism, Condition (iv) becomes 
$$
c_2 < \tfrac{1}{2}c_1(c_1-K) + 2\chi + (Hc_1 - \tfrac{1}{2} \rho HK) \ell + \tfrac{1}{2} \rho H^2 \ell^2,
$$ 
so the upper bound on $c_2$ can be made \emph{arbitrarily large} by taking $\ell \gg 0$. Therefore, Condition (iv) cannot be dropped, but it can always be made to be satisfied.
\end{remark}

Next, we prove Theorem \ref{thminv} on topological invariance of canonical virtual Segre and Verlinde numbers from the introduction.
\begin{proof}[Proof of Theorem \ref{thminv}]
We give the proof for canonical virtual Segre numbers. The case of canonical virtual Verlinde numbers is similar. The case $\rho=1$ follows at once from Theorem \ref{MOPthm1} and the fact that $\chi$ and $K^2$ can be expressed in terms of $e(S)$ and $\sigma(S)$, so we take $\rho \in \Z_{>1}$. Let $S$ be a smooth projective surface satisfying $b_1(S) = 0$ and $K$ very ample. Then $S$ is minimal of general type, so its only Seiberg-Witten basic classes are $0$ and $K \neq 0$, and $\SW(0)=1$, $\SW(K) = (-1)^{\chi}$ \cite[Thm.~7.4.1]{Mor}. Suppose furthermore $H = K$. 
For any $\boldsymbol{\lambda} = (\lambda_1, \ldots, \lambda_m) \in \Z^m, \boldsymbol{\mu} = (\mu_1, \ldots, \mu_m) \in \Z^m$, we write
$$
\boldsymbol{\lambda} \O(K)^{ \otimes \boldsymbol{\mu}} := \sum_{i=1}^{m} \lambda_i [\O(K)^{\otimes \mu_i}] \in K(S).
$$
For any $\boldsymbol{\lambda}, \boldsymbol{\mu} \in \Z^m, n \in \Z$, consider the following generating function of canonical virtual Segre numbers of $S$
$$
\mathsf{G}_{\rho,\boldsymbol{\lambda},\boldsymbol{\mu},n,S} := \sum_{\ell,c_2 \in \Z \atop \gcd(\rho,\ell K^2) = 1} w^\ell z^{\frac{\vd(M)}{2}} \int_{[M]^{\vir}} c((\boldsymbol{\lambda} \O(K)^{ \otimes \boldsymbol{\mu}})_M) \, e^{\mu(\O(K)^{\otimes n}) + \mu(\pt) u},
$$
where $M:=M_S^{K}(\rho,\ell K,c_2)$ and $\vd(M) = 2\rho c_2 - (\rho-1) \ell^2 K^2 - (\rho^2 - 1) \chi$. The condition $\gcd(\rho,\ell K^2) = 1$ implies that Gieseker and $\mu$-stability coincide and there are no strictly semistable sheaves in $M$ for any $c_2$ \cite[Lem.~1.2.13, 1.2.14]{HL}. Moreover, tensoring by multiples of $K$ is an isomorphism on these moduli spaces and does not change the integrands, so it suffices to only sum over 
$$
\ell \in \{N, N+1, \ldots, N+\rho-1\} \quad \mathrm{such \, that} \quad \gcd(\rho,\ell K^2)=1,
$$
for \emph{any} fixed $N \in \Z$, and consider the powers of $w$ as elements of $\Z / \rho \Z$. 

Consider the conditions of Theorem \ref{structhm1}. Condition (i), (iii), and (v) are automatically satisfied (taking $N \gg 0$) and (ii) can be dropped by Remark \ref{univdrop1}. Taking $N \rightarrow \infty$, the upper bound on $c_2$ in Condition (iv) becomes arbitrarily large. We conclude that $\mathsf{G}_{\rho,\boldsymbol{\lambda},\boldsymbol{\mu},n,S}$, modulo an arbitrarily large power of $z$, is determined by a universal function which only depends on $\rho$, $\boldsymbol{\lambda}$, $\boldsymbol{\mu}$, $n$, $\chi$, $K^2$ (recall that the only Seiberg-Witten basic classes of $S$ are $0$ and $K \neq 0$). Since $\chi$, $K^2$ can be expressed in terms of $e(S)$, $\sigma(S)$, the result follows.
\end{proof}

\subsection{Main conjectures} \label{sec:generalconj}

In Section \ref{sec:mainthms}, we mostly worked with arbitrary smooth projective surfaces $S$ satifying $p_g(S)>0$ and $b_1(S)=0$. Accordingly, we formulate more general versions of Conjectures \ref{conj1} and \ref{conj2}, for which we present verifications in Section \ref{sec:verif} by applying Theorems \ref{structhm1} and \ref{structhm2}. 
\begin{conjecture} \label{conj1strong}
Let $\rho  \in \Z_{>0}$ and $s \in \Z$. There exist $V_s$, $W_s$, $X_s$, $Q_s$, $R_s$, $T_s  \in \C[[z]]$, $Y_s$, $Z_s$, $Y_{j,s}$, $Z_{jk,s}$, $S_s$, $S_{j,s} \in \C[[z^{\frac{1}{2}}]]$ for all $1 \leq j \leq k \leq \rho-1$ with the following property.\footnote{These universal power series depend on $\rho$ and $s$. We suppress the dependence on $\rho$.} Let $(S,H)$ be a smooth polarized surface satisfying $b_1(S) = 0$ and  $p_g(S)>0$. Suppose $M:=M_S^H(\rho,c_1,c_2)$ contains no strictly semistable sheaves. For any $\alpha \in K(S)$ with $\rk(\alpha)=s$ and $L \in \Pic(S)$, $\int_{[M]^{\vir}} c( \alpha_M ) \, \exp(\mu(L)  + \mu(\pt) u)$ equals the coefficient of $z^{\frac{1}{2} \vd(M)}$ of\footnote{We stress that here (and similarly elsewhere), the coefficients $j$ and $k$ of $Z_{jk,s}$ are placed adjacent (and not multiplied). They are not separated by another comma in order to avoid cluttered notation.}
\begin{align*}
&\rho^{2 - \chi+K^2} \, V_s^{c_2(\alpha)} \, W_s^{c_1(\alpha)^2} \, X_s^{\chi} \, Y_{s}^{c_1(\alpha) K} \, Z_{s}^{K^2} \, e^{L^2 Q_s + (c_1(\alpha)L) R_s + (LK) S_{s}  + u \, T_s} \\
&\cdot \sum_{(a_1, \ldots, a_{\rho-1})} \prod_{j=1}^{\rho-1} \eps_{\rho}^{j a_j c_1} \, \SW(a_j) \, Y_{j,s}^{c_1(\alpha) a_j} \, e^{(a_j L) S_{j,s}} \prod_{1 \leq j \leq k \leq \rho-1} Z_{jk,s}^{a_j a_k},
\end{align*}
where the sum is over all $(a_1, \ldots, a_{\rho-1}) \in H^2(S,\Z)^{\rho-1}$. 
Moreover
\begin{align*}
V_s(z) &= (1+(1-\tfrac{s}{\rho})t)^{1-s} (1+(2-\tfrac{s}{\rho})t)^s (1+(1-\tfrac{s}{\rho})t)^{\rho-1} , \\ 
W_s(z) &= (1+(1-\tfrac{s}{\rho})t)^{\frac{1}{2}s-1} (1+(2-\tfrac{s}{\rho})t)^{\frac{1}{2}(1-s)} (1+(1-\tfrac{s}{\rho})t)^{\frac{1}{2} - \frac{1}{2} \rho}, \\
X_s(z) &=  (1+(1-\tfrac{s}{\rho})t)^{\frac{1}{2} s^2-s} (1+(2-\tfrac{s}{\rho})t)^{-\frac{1}{2}s^2+\frac{1}{2}} (1+(1-\tfrac{s}{\rho})(2-\tfrac{s}{\rho})t)^{-\frac{1}{2}} (1+(1-\tfrac{s}{\rho})t)^{-\frac{(\rho-1)^2}{2\rho} s}, \\
Q_s(z) &= \tfrac{1}{2}t(1+(1-\tfrac{s}{\rho}) t), \quad R_s(z) = t, \quad T_s(z) = \rho t (1+ \tfrac{1}{2}(1-\tfrac{s}{\rho})(2-\tfrac{s}{\rho})t),
\end{align*}
where $$z = t (1+(1-\tfrac{s}{\rho}) t)^{1-\frac{s}{\rho}}.$$ Furthermore, $Y_s$, $Y_{j,s}$,  $Z_s$, $Z_{jk,s}$, $S_s$, $S_{j,s}$ are all algebraic functions.
\end{conjecture}

\begin{conjecture} \label{conj2strong}
Let $\rho \in \Z_{>0}$ and $r \in \Z$. There exist $G_r$, $F_{r} \in \C[[w]]$, $A_{r}$, $B_{r}$, $A_{j,r}$, $B_{jk,r} \in \C[[w^{\frac{1}{2}}]]$ for all $1 \leq j \leq k \leq \rho-1$ with the following property. Let $(S,H)$ be a smooth polarized surface satisfying $b_1(S) = 0$, $p_g(S)>0$, and let $L \in \Pic(S)$. Suppose $M:=M_S^H(\rho,c_1,c_2)$ contains no strictly semistable sheaves. Then $\chi^{\vir}(M, \mu(L) \otimes E^{\otimes r})$ equals the coefficient of $w^{\frac{1}{2}\vd(M)}$ of
\begin{align} \label{Verlindenum}
\rho^{2 - \chi+K^2} \, G_{r}^{\chi(L)} \, F_{r}^{\frac{1}{2} \chi} \, A_{r}^{LK} \, B_{r}^{K^2} \sum_{(a_1, \ldots, a_{\rho-1})} \prod_{j=1}^{\rho-1} \eps_{\rho}^{j a_j c_1} \, \SW(a_j) \, A_{j,r}^{a_j L} \prod_{1 \leq j \leq k \leq \rho-1} B_{jk,r}^{a_j a_k},
\end{align}
where the sum is over all $(a_1, \ldots, a_{\rho-1}) \in H^2(S,\Z)^{\rho-1}$. Furthermore, $A_r$, $B_r$, $A_{j,r}$, $B_{jk,r}$ are all algebraic functions.
\end{conjecture}

\begin{remark}
These conjectures were partially inspired by the universal formulae \emph{before} taking $\mathrm{Coeff}_{t_{1}^0} \cdots \mathrm{Coeff}_{t_{r-1}^0}$ in Theorems \ref{structhm1}, \ref{structhm2}, and similar universal formulae for Vafa-Witten invariants \cite{GKL}.
\end{remark}

These conjectures imply Conjectures \ref{conj1} and \ref{conj2} from the introduction as follows. For all $\rho \in \Z_{>1}$ and any, possibly empty, subset $J\subset [\rho-1]:=\{1,\ldots,\rho-1\}$, define
\begin{align}
\begin{split} \label{Jseries}
Y_{J,s}&:=Y_s\prod_{j\in J}Y_{j,s}, \quad Z_{J,s}:=Z_s\prod_{i\le j\in J}Z_{ij,s},\quad S_{J,s}:=S_s+\sum_{j\in J}S_{j,s}, \\
 A_{J,r}&:=A_r\prod_{j\in J}A_{j,r}, \quad B_{J,r}:=B_r\prod_{i\le j\in J}B_{ij,r}.\
 \end{split}
 \end{align}
Suppose $S$ is a smooth projective surface $S$ with $b_1(S)=0$, $p_g(S)>0$, and suppose its only Seiberg-Witten basic classes are $0$ and $K \neq 0$ (e.g.~minimal surfaces $S$ of general type satisfying $b_1(S)=0$ and $p_g(S)>0$ \cite[Thm.~7.4.1]{Mor}). Then $\SW(0) = 1$ and $\SW(K) = (-1)^{\chi}$ and the formulae in the conjectures of the introduction follow.

We note that knowing the power series $Y_{J,s}$, $Z_{J,s}$, $S_{J,s}$, $A_{J,r}$, $B_{J,r}$ is equivalent to knowing the power series
$Y_s$, $Y_{i,s}$, $Z_s$, $Z_{ij,s}$, $S_s$, $S_{i,s}$, $A_r$, $A_{i,r}$, $B_r$, $B_{ij,r}$. E.g.~for $Z_{J,s}$, indeed $J = \varnothing$ determines $Z_s$, $J = \{i\}$ then determines $Z_{ii,s}$, and $J = \{i<j\}$ then determines $Z_{ij,s}$.
We provide numerous conjectural formulae (and verifications) for these power series in Section \ref{sec:alg}.

\section{Verifications} \label{sec:verif}

\subsection{Rank 1} \label{sec:verif:rk1}

We start with two propositions, which can be seen as (much easier) rank 1 analogs of the Witten conjecture for $\mathrm{SU}(2)$ Donaldson invariants.
\begin{proposition}  \label{rk1Witten}
For any smooth projective surface $S$ and $L \in \Pic(S)$, we have
$$
\sum_{n=0}^{\infty} z^n \int_{S^{[n]}} e^{\mu(L)  + \mu(\pt) u} = e^{(\frac{1}{2} L^2 + u)z}.
$$
\end{proposition}
\begin{proof}
The universal sheaf $\E$ on $S\times S^{[n]}$ is the ideal sheaf $\I_{\cZ}$ of the universal subscheme $\cZ \subset S\times S^{[n]}$. Therefore $c_2(\E)= [\cZ ]$, i.e.~the (Poincar\'e dual of the) fundamental class of $\cZ$.
Consider the symmetric product $S^{(n)}$ and the Hilbert-Chow morphism
$$
\rho : S^{[n]} \rightarrow S^{(n)}
$$
and denote the universal $2n$-cycle by $\mathcal{D} \subset S \times S^{(n)}$. Analogous to the $\mu$-insertion $\mu(\sigma)= \pi_{S^{[n]}*}( \pi_S^*\sigma \cdot  [\mathcal{Z}] )$, we define
\begin{equation*}
\tau(\sigma):= \pi_{S^{(n)}*}\big( \pi_S^*\sigma \cdot  [\mathcal{D}] \big) \in H^*(S^{(n)},\Q),
\end{equation*}
for all $\sigma \in H^*(S,\Q)$. Note that $(1_S \times \rho)_* [\mathcal{Z}] = [\mathcal{D}]$, or more generally
\begin{equation} \label{mutotau}
(1_{S^{\ell}} \times \rho)_* \big( \pi_{1,\ell+1}^* [\cZ] \cdots \pi_{\ell,\ell+1}^* [\cZ] \big) = \pi_{1,\ell+1}^* [\mathcal{D}] \cdots \pi_{\ell,\ell+1}^* [\mathcal{D}], 
\end{equation}
where $\pi_{i,\ell+1} : S^{\ell} \times S^{[n]} \rightarrow S \times S^{[n]}$ denotes the projection onto the factors $(i,\ell+1)$ and similarly on $S^{\ell} \times S^{(n)}$. Therefore
\begin{equation} \label{pushrho}
\rho_* (\mu(\sigma_1) \cdots \mu(\sigma_{\ell})) = \tau(\sigma_1) \cdots \tau(\sigma_{\ell}),
\end{equation}
for all $\sigma_1, \ldots, \sigma_{\ell} \in H^*(S,\Q)$. Hence
\begin{align*}
\int_{S^{[n]}} e^{\mu(L)  + \mu(\pt) u} = \int_{S^{(n)}} e^{\tau(L)  + \tau(\pt) u} .
\end{align*}

Next, consider the natural degree $n!$ morphism
$$
\eps : S^{n} \rightarrow S^{(n)}.
$$
On $S \times S^{n}$, we consider the $2n$-cycle 
$$
\Delta := \sum_{i=2}^{n+1} \Delta_{1i},
$$
where $\Delta_{1i} := \{ p \in S \times S^{n} : p_1 = p_i\}$. Analogous to the $\mu$- and $\tau$-insertions, we define
\begin{equation*}
\nu(\sigma):= \pi_{S^{n}*}\big( \pi_S^*\sigma \cdot  [\Delta] \big) = \pi_1^* \sigma + \cdots + \pi_{\ell}^* \sigma \in H^*(S^{n},\Q),
\end{equation*}
for all $\sigma \in H^*(S,\Q)$ and where $\pi_ i : S^{\ell} \rightarrow S$ denotes projection on the $i$th component. Using $(1_S \times \eps)_* [\Delta] =  n! [\mathcal{D}]$ (and the analog of \eqref{mutotau}), we find
$$
\eps_* (\nu(\sigma_1) \cdots \nu(\sigma_{\ell})) = n! \tau(\sigma_1) \cdots \tau(\sigma_{\ell}).
$$
Therefore 
\begin{align*}
\int_{S^{(n)}} e^{\tau(L)  + \tau(\pt) u} &= \frac{1}{n!} \int_{S^{n}} e^{\nu(L)  + \nu(\pt) u} \\
&= \frac{1}{n!} \int_{S^{n}} e^{\pi_1^*(L + \pt \, u) + \cdots + \pi_n^*(L + \pt \, u) } \\
&= \frac{1}{n!} \Big(\int_S e^{L +  \pt \, u}\Big)^n = \frac{1}{n!} \Big( \frac{1}{2} L^2 + u \Big)^n. \qedhere
\end{align*}
\end{proof}

\begin{proposition} \label{rk1Wittenmatter}
For any smooth projective surface $S$ and $\mathcal{L}, L \in \Pic(S)$, we have
$$
\sum_{n=0}^{\infty} z^n \int_{S^{[n]}} c(\mathcal{L}^{[n]})  \, e^{\mu(L)  + \mu(\pt) u} = e^{(\frac{1}{2} L^2 + \mathcal{L} L + u)z}.
$$
\end{proposition}
\begin{proof}
For any $n$, we denote the Hilbert-Chow morphism by
$$
\rho : S^{[n]} \rightarrow S^{(n)}.
$$
Let $\mathbb{H} := \bigoplus_{n \geq 0} H^*(S^{[n]},\Q)$. For any $\sigma \in H^*(S,\Q)$ and $i>0$, we recall the definition of the Nakajima creation operator \cite{Nak}
$$
\mathfrak{q}_{i}(\sigma) : \mathbb{H} \rightarrow \mathbb{H}
$$
defined on $H^*(S^{[n]},\Q)$ by the correspondence
$$
\mathfrak{q}_{i}(\sigma)(-) = \pi_{S^{[n+i]} *} \Big( \pi_S^* \sigma \cdot \pi_{S^{[n]}}^*(-) \cdot [\mathcal{Z}^{n,n+i}] \Big), 
$$
where
\begin{displaymath}
\xymatrix
{
& S^{[n]} \times S \times S^{[n+i]} \ar_{\pi_{S^{[n]}}}[dl] \ar^{\pi_S}[d] \ar^{\pi_{S^{[n+i]}}}[dr] & \\
S^{[n]} & S & S^{[n+i]}
}
\end{displaymath}
and $\mathcal{Z}^{n,n+i} \subset S^{[n]} \times S \times S^{[n+i]}$ is the incidence locus of triples $(Z,p,Z')$ satisfying $Z \subset Z'$ and $\rho(Z') = \rho(Z) + ip$. Then \cite[Thm.~4.6]{Leh} 
\begin{equation}\label{lech}
\sum_{n=0}^{\infty} c(\mathcal{L}^{[n]}) z^n = \exp\Big( \sum_{i=1}^{\infty} \frac{(-1)^{i-1}}{i} \mathfrak{q}_{i}(c(\mathcal{L})) z^i \Big) \mathbbm{1},
\end{equation}
where $\mathbbm{1}$ denotes the fundamental class of $S^{[0]} = \mathrm{pt}$.

We denote by $\mathcal{D}^{n,n+i} \subset S^{(n)} \times S \times S^{(n+i)}$  the incidence locus of triples $(Z,p,Z')$ satisfying $Z' = Z+ip$. Note that the natural projection  $\mathcal{D}^{n,n+i} \to  S^{(n)} \times S$ is an isomorphism. 
On $\bigoplus_{n \geq 0} H^*(S^{(n)},\Q)$ we consider operators $\mathfrak{p}_{i}(\sigma)$ defined on $H^*(S^{(n)},\Q)$ by the correspondence
$$
\mathfrak{p}_{i}(\sigma)(-) = \pi_{S^{(n+i)} *} \Big( \pi_S^* \sigma \cdot \pi_{S^{(n)}}^*(-) \cdot [\mathcal{D}^{n,n+i}] \Big).
$$
Consider the commutative diagram
\begin{displaymath}
\xymatrix
{
& \mathcal{Z}^{n,n+i} \ar@{->>}_{\pi_{S^{[n]}} \times \pi_S}[dl] \ar^{\rho \times 1_S \times \rho}[d] \ar@/^/[dr]^{\pi_{S^{[n+i]}}}&  \\
S^{[n]} \times S \ar_{\rho \times 1_S}[dr] &  \mathcal{D}^{n,n+i} \ar^{\cong}_\pi[d] \ar^{\pi_{S^{(n+i)}}}[dr]&S^{[n+i]}\ar^{\rho}[d] \\
& S^{(n)} \times S\ar[r]&S^{(n+i)}.
}
\end{displaymath}
The general fibre of the map $\rho\times 1_S\times \rho:\mathcal{Z}^{n,n+i}\to\mathcal{D}^{n,n+i}$ is irreducible of dimension $i-1$ (Brian\c{c}on). For any  $W\in H_k(S^{[n]}\times S,\Q)$, we have $ (\pi_{S^{[n]}}\times\pi_S)^*(W)\in H_{k+2i-2}(\mathcal{Z}^{n,n+i},\Q )$. However, the dimension of the support of $(\rho\times 1_S\times \rho)_*(\pi_{S^{[n]}}\times\pi_S)^*(W)$ is at most $k$. Thus for any $i>1$, we get that
$$
(\rho\times 1_S\times \rho)_*\circ (\pi_{S^{[n]}}\times\pi_S)^*: H^*(S^{[n]}\times S,\Q)\to H^*(\mathcal{D}^{n,n+i},\Q)
$$ 
is the zero map. 
This gives 
$
\rho_* \big(\mathfrak{q}_{i}(\sigma)(\beta) \big) = 0,
$
for all  $i>1$, $\sigma\in H^*(S,\Q)$, $\beta \in H^*(S^{[n]},\Q)$. 

As the creation operators $\mathfrak{q}_{i}(\sigma_i)$ commute, this implies that 
$$
\rho_* \big(\mathfrak{q}_{i_1}(\sigma_1)\cdots \mathfrak{q}_{i_\ell}(\sigma_\ell)\mathbbm{1} \big) = 0,
$$
for all $\sigma_1, \ldots, \sigma_{\ell} \in H^*(S,\Q)$ and $i_1, \ldots, i_\ell \geq 1$ with at least one $i_k>1$. 
Therefore we get from \eqref{lech} that
$\rho_*c(\mathcal{L}^{[n]})=\rho_{*}\exp(\mathfrak{q}_1(\mathcal{L}))\mathbbm{1}$.

On the other hand, in case $i=1$, let  $\sigma\in H_*(S,\Q)$, $\beta\in H_*(S^{(n)},\Q)$, then the map $\rho\times 1_S\times \rho$ restricted to the inverse image in $\mathcal{Z}^{n,n+1}$ of the support of a cycle $\beta\times \sigma$ is generically one-to-one. Thus
 $$
(\rho\times 1_S\times \rho)_*(\pi_{S^{[n]}}\times\pi_S)^*(\rho\times 1_S)^*(\beta\times \sigma) =\pi^*(\beta\times \sigma).
$$
This gives $\rho_*(\mathfrak{q}_1(\sigma)\rho^*(\beta))=\mathfrak{p}_1(\sigma)\beta$, and thus inductively $\rho_{*}((\mathfrak{q}_1(\sigma))^n \mathbbm{1})=(\mathfrak{p}_1(\sigma))^n \mathbbm{1}$.
Note that for $\sigma\in H^*(S,\Q)$ we have $\mu(\sigma)=\rho^* \tau(\sigma)$.
Therefore, by \eqref{lech} and the projection formula, we get
\begin{align*}
\int_{S^{[n]}} c(\mathcal{L}^{[n]})  \, e^{\mu(L)  + \mu(\pt) u}  &= \frac{1}{n!} \int_{S^{[n]}}  e^{\mu(L)  + \mu(\pt) u}  (\mathfrak{q}_1(c(\mathcal{L})))^n \mathbbm{1}\\
&=\frac{1}{n!} \int_{S^{(n)}}  e^{\tau(L)  + \tau(\pt) u}  (\mathfrak{p}_1(c(\mathcal{L})))^n \mathbbm{1}.
\end{align*}
Finally, using the morphism $\eps : S^n \rightarrow S^{(n)}$, we have
$$
(\mathfrak{p}_1(\sigma))^n \mathbbm{1} = \eps_* \big( \pi_1^* \sigma \cdots \pi_n^* \sigma \big),
$$ 
for all $\sigma \in H^*(S,\Q)$, where $\pi_i : S^n \rightarrow S$ denotes projection on the $i$th component. Using the notation of the proof of Lemma \ref{rk1Witten}, we have $\nu(\sigma) = \eps^* \tau(\sigma)$, for all $\sigma \in H^*(S,\Q)$. This implies
\begin{align*}
\int_{S^{[n]}} c(\mathcal{L}^{[n]})  \, e^{\mu(L)  + \mu(\pt) u}  &= \frac{1}{n!} \int_{S^n} (\pi_1^* c(\mathcal{L}) \cdots \pi_n^* c(\mathcal{L})) \, e^{\nu(L)  + \nu(\pt) u} \\
&= \frac{1}{n!} \int_{S^n} (\pi_1^* c(\mathcal{L}) \cdots \pi_n^* c(\mathcal{L})) \, e^{\pi_1^*(L + \pt \, u) + \cdots + \pi_n^*(L +  \pt \, u)} \\
&= \frac{1}{n!} \Big( \int_{S} c(\mathcal{L}) e^{L + \pt \, u} \Big)^n = \frac{1}{n!} \Big( \frac{1}{2} L^2 + \mathcal{L} L + u \Big)^n. \qedhere
\end{align*}
\end{proof}

We summarize what is known for rank 1 Segre integrals:
\begin{theorem}[Marian-Oprea-Pandharipande+$\varepsilon$] \label{rk1prop}
There exist $V_s$, $W_s$, $X_s$, $Y_s$, $Z_s \in 1+z \, \Q[[z]]$, $Q_s$, $R_s$, $S_s$, $T_s \in z \Q[[z]]$ with the following properties. Let $S$ be a smooth projective surface, $\alpha \in K(S)$ such that $\rk(\alpha) = s$, and $L \in \Pic(S)$. Then
\begin{equation*} 
\sum_{n=0}^{\infty} z^n \int_{S^{[n]}} c( \alpha^{[n]} ) \, e^{\mu(L)  + \mu(\pt) u} = V_s^{c_2(\alpha)} \, W_s^{c_1(\alpha)^2} \, X_s^{\chi} \, Y_{s}^{c_1(\alpha) K} \, Z_{s}^{K^2} \, e^{L^2 Q_s + (c_1(\alpha)L) R_s + (LK) S_{s}  + u \, T_s}.
\end{equation*}
By \cite{MOP3}, under the formal change of variables $z = t(1+(1-s)t)^{1-s}$, we have 
\begin{align*}
V_s(z) &= (1+(1-s)t)^{1-s} (1+(2-s)t)^s, \\ 
W_s(z) &= (1+(1-s)t)^{\frac{1}{2}s-1} (1+(2-s)t)^{\frac{1}{2}(1-s)}, \\
X_s(z) &= (1+(1-s)t)^{\frac{1}{2} s^2-s} (1+(2-s)t)^{-\frac{1}{2}s^2+\frac{1}{2}}(1+(2-s)(1-s)t)^{-\frac{1}{2}}.
\end{align*}
Also by \cite{MOP3}, under the same change of variables, $Y_s$ is determined for $s \in \{-2,-1,1,2\}$ and $Z_s$ for $s \in \{-2,-1,0,1,2\}$. 
Furthermore
\begin{align*}
&Q_0(z) = \tfrac{1}{2} z, \quad S_0(z) = 0, \quad  T_0(z) = z, \\
&Q_1(z) = \tfrac{1}{2} z, \quad R_1(z) = z, \quad S_1(z) = 0, \quad T_1(z) = z.
\end{align*}
\end{theorem}
\begin{proof}
Universality follows from (a more elementary analog of) Steps 1--3 in the proof of Theorem \ref{structhm2}. Setting $L = u =0$, the statements about $V_s, W_s, X_s, Y_s, Z_s$ follow from \cite{MOP3}. Furthermore, $X_0(z) = X_1(z) = Z_0(z) = Z_1(z) = 1$ \cite{MOP3}. Hence the rest of the theorem follows from Propositions \ref{rk1Witten} and \ref{rk1Wittenmatter}.
\end{proof}

\begin{remark} \label{rk1basis}
We verified the formulae for $Q_s, R_s, T_s$ of Conjecture \ref{conj1} for $\rho=1$ modulo $z^{11}$.
For this, we choose a collection of triples $(S,\alpha,L)$ such that the corresponding vectors 
$$
(c_2(\alpha), c_1(\alpha)^2, \chi, c_1(\alpha) K, K^2, L^2, c_1(\alpha)L, LK)
$$
are $\Q$-linearly independent. Taking $S$ a toric surface with torus $T$, we obtain a lift of the $T$-action to $S^{[n]}$ with isolated reduced fixed points indexed by collections of monomial ideals.
By taking $\alpha$ and $L$ with $T$-equivariant structure, one can calculate the Segre numbers with $\mu$-insertions for $(S,\alpha,L)$ by Atiyah-Bott localization. Specifically, we choose 
$
(S,\alpha,L) = (\PP^2,\O^{\oplus s},\O)$, $(\PP^1\times \PP^1,\O^{\oplus s},\O)$, $(\PP^2,\O(1)\oplus \O^{\oplus s-1},\O)$, $(\PP^2,\O^{\oplus s},\O(1)))$, $(\PP^2,\O(1)\oplus \O^{\oplus s-1},\O(1))$, $(\PP^2,\O(1)^{\oplus 2}\oplus \O^{\oplus s-2},\O)$,
 $(\PP^1\times \PP^1,\O(0,1)\oplus \O^{\oplus s-1},\O)$, $(\PP^1\times \PP^1,\O^{\oplus{s}},\O(0,1))$ and calculate their Segre numbers with $\mu$-insertion up to order 10. 
\end{remark}

\subsection{Higher rank: calculations} \label{sec:verif:higherrk}

We want to explicitly calculate the first few coefficients of the universal power series of Theorems \ref{structhm1}, \ref{structhm2} for ranks $\rho=2,3,4$. The definition of the generating function
\begin{equation*} 
\mathsf{Z}_S(\alpha,L,\boldsymbol{a} ,\boldsymbol{t},u,q)
\end{equation*}
in \eqref{defZ} makes sense for \emph{any} smooth projective surface $S$ and \emph{any} $\alpha \in K(S)$, $L \in \Pic(S)$, and $\boldsymbol{a}=(a_1, \ldots,a_{\rho}) \in A^1(S)^\rho$. Consider any finite collection $\mathcal{C}$ containing $(S,\alpha,L,\boldsymbol{a})$ such that the corresponding vectors
\begin{align*}
(c_1(\alpha)^2,  c_1(\alpha)L, c_1(\alpha)K, c_2(\alpha), L^2, LK, K^2, \chi, \{a_ic_1(\alpha)\}, \{a_iL\}, \{a_iK\}, \{a_ia_j\})
\end{align*}
are $\Q$-linearly independent. Then the universal function of Theorem \ref{structhm1} is determined by $\mathsf{Z}_S(\alpha,L,\boldsymbol{a} ,\boldsymbol{t},u,q)$ on this finite collection $\mathcal{C}$ via equation \eqref{ZA}. 

Now take $(S,\alpha,L,\boldsymbol{a})$ such that $S$ is a toric surface with torus $T$ and $\alpha$, $L$, $\boldsymbol{a}$ are $T$-equivariant. The action of $T$ on $S$ lifts to $S^{[\boldsymbol{n}]}$ for any $\boldsymbol{n} = (n_1, \ldots, n_{\rho}) \in \Z_{\geq 0}^{\rho}$. Therefore, we can apply the Atiyah-Bott localization formula to explicitly determine $\mathsf{Z}_S(\alpha,L,\boldsymbol{a} ,\boldsymbol{t},u,q)$ up to some order in $q$. We carried this out for $S = \PP^2$, $\PP^1\times \PP^1$ and certain choices of $\alpha$, $L$, $\boldsymbol{a}$ similar to Remark \ref{rk1basis}. The reader can consult \cite{GK1, GK2, GK3, GKW, Laa1} for more details on Atiyah-Bott calculations in closely related settings. This discussion holds analogously in the Verlinde case of Theorem \ref{structhm2}. 

We determined the universal functions of Theorems \ref{structhm1}, \ref{structhm2} up to the following orders: 
\begin{itemize}
\item \textbf{Rank $\rho=2$.} Keeping $s$ as a variable, we determined $A_{\bullet,s}(q)$ (i.e.~$A_{J,s}$ for all $J \subset [\rho-1]$) modulo $q^{11}$. For $s=5,6$ we determined $A_{\bullet,s}(q)$ modulo $q^{26}$. Keeping $r$ as a variable, we determined $B_{\bullet,r}(q)$ modulo $q^{16}$.
\item \textbf{Rank $\rho=3$.} For $s \in \{-3, \ldots, 12\}$, we determined $A_{\bullet,s}(q)$ modulo $q^{10}$. For $s=5,6$ we determined $A_{\bullet,s}(q)$ modulo $q^{26}$. For $r \in \{-11,\ldots, 3\}$, we determined $B_{\bullet,r}(q)$ modulo $q^{9}$.
\item \textbf{Rank $\rho=4$.} For $s \in \{0, \ldots, 8\}$, we determined $A_{\bullet,s}(q)$ modulo $q^{8}$. For $s=-1$, we also determined $A_{\underline{a_4^2},s}, A_{\underline{a_4 K},s}, A_{\underline{K^2},s}, A_{\underline{\chi},s}$ modulo $q^8$. For $\rho=4$, the Verlinde calculations are harder and we determined no coefficients of the $B_{\bullet,r}(q)$. 
\end{itemize}
With this data, we can verify Conjectures \ref{conj1strong}, \ref{conj2strong}, and \ref{conj3} in the following cases (always for certain values of $H,c_1$ such that there are no strictly semistable sheaves): \\

\noindent \textbf{Rank $\rho=2$}.
\begin{itemize}
\item Conjecture \ref{conj1strong} holds for $S$ a $K3$ surface and virtual dimension up to 16, for $S$ an elliptic surface\footnote{An elliptic surface of type $E(n)$ is an elliptic surface $S \rightarrow \PP^1$ with section, $12n$ rational 1-nodal fibres, and no further singular fibres.} of types $E(3)$, $E(4)$, $E(5)$ up to virtual dimension $16$, for $S$
a double cover of $\PP^2$ branched along a smooth octic up to virtual dimension $14$, for $S$ a double cover of $\PP^1\times \PP^1$ branched along a smooth curve of bidegree $(6,6)$ up to virtual dimension $14$, and 
for $S$ a general quintic in $\PP^3$ up to virtual dimension $12$. 
Conjecture  \ref{conj1strong} also holds for $S$  the blow-up of one of the above surfaces in a point, with the same bounds on the virtual dimension. 
Conjecture \ref{conj1strong} also holds for $S$ an elliptic surface of type $E(3)$ up to virtual dimension 18.
\item Conjecture \ref{conj2strong}  holds for $S$ a $K3$ surface and virtual dimension up to $18$, for the blow-up of a $K3$ surface in a point up to virtual dimension $13$, for $S$ an elliptic surface of type $E(3)$
up to virtual dimension $18$, for $S$ an elliptic surface of type $E(4)$ up to virtual dimension $12$, for $S$ an elliptic surface of type $E(5)$ up to virtual dimension $10$, for $S$
a double cover of $\PP^2$ branched along a smooth octic up to virtual dimension $12$, and 
for $S$ a general quintic in $\PP^3$ up to virtual dimension $10$.
\item Conjecture \ref{conj3} holds up to virtual dimension $18$.
\end{itemize}

\noindent \textbf{Rank $\rho=3$}.
\begin{itemize}
\item Let $s\in\{-3,\ldots,12\}$. Then Conjecture \ref{conj1strong} holds for $S$ a $K3$ surface up to virtual dimension $14$, for $S$ the blow-up of a $K3$ surface in a point up to virtual dimension $14$, for $S$ an elliptic surface of type $E(3)$ up to virtual dimension $12$, for $S$ a double cover of $\PP^2$ branched along a smooth octic up to virtual dimension $6$. Conjecture \ref{conj1strong} also holds with the same dimension bounds for blow-ups of these surfaces in one point. 
\item Let $r\in\{-11,\ldots,3\}$. Then Conjecture \ref{conj2strong} holds for $S$ a $K3$ surface up to virtual dimension $12$, for $S$ the blow-up of a $K3$ surface in a point up to virtual dimension $12$, for $S$ an elliptic surface of type $E(3)$ up to virtual dimension $8$, and for $S$ a double cover of $\PP^2$ branched along a smooth octic up to virtual dimension $6$. Conjecture \ref{conj2strong} also holds with the same dimension bounds for blow-ups of these surfaces in one point. 
\item Conjecture \ref{conj3} holds for $s\in\{-3,\ldots,6\}$, $r\in\{-6,\ldots,3\}$ up to virtual dimension $12$.
\end{itemize}

We expect that Conjectures \ref{conj1} and \ref{conj2} hold for all ``virtual surfaces'' satisfying 
$$
2 - \chi + K^2 \geq 0.
$$ 
This inequality ensures that the first term in our conjectural formulae is integer. By this we mean one formally calculates the virtual Segre and Verlinde numbers using Theorems \ref{structhm1} and \ref{structhm2} for values of $K^2 ,\chi$ for which there exist no minimal general type surfaces $S$ satisfying $b_1(S) = 0$ and $p_g(S)>0$ with these values of $K^2 ,\chi$, but one nonetheless obtains the numbers given by the conjectures. \\

\noindent \textbf{Rank $\rho=4$}. 
\begin{itemize}
\item Let $s \in \{0,\ldots,8\}$. Then Conjecture \ref{conj1strong} holds for $S$ a $K3$ surface up to virtual dimension $6$ and for $S$ the blow-up of a $K3$ surface in a point up to virtual dimension $6$. The following can be seen as indirect evidence: Conjecture \ref{conj1} holds for $S$ a ``virtual'' surface with $K^2=-1$, $\chi=0$ up to virtual dimension $9$ and for $S$ a ``virtual'' surface with
$K^2=-1$, $\chi=1$ up to virtual dimension $8$. 
\item Let $S$ be a $K3$ surface, $c_1^2=8$, $\alpha = -[\O_S]$ (so $s = -1$), and $L=u=0$. Then Conjecture \ref{conj1strong} holds for $c_2=7$ (yielding virtual Segre number $\tfrac{15}{4}$).\footnote{This case ``probes'' the power $-\frac{(\rho-1)^2}{2\rho} s = \tfrac{9}{8}$ in $X_{s}(z)$.} 
\end{itemize}

The computations support a further conjecture on the dependence  of the universal power series in Conjecture \ref{conj1strong} and \ref{conj2strong} on $s$ and $r$.
\begin{conjecture}\label{polconj} 
\begin{enumerate}
\item For all $n \in \Z_{\geq 0}$, the coefficient of $z^{\frac{n}{2}}$ in the universal power series $V_s$, $W_s$, $X_s$, $Y_s$, $Y_{j,s}$, $Z_s$, $Z_{ij,s}$ $Q_s$, $R_s$, $S_s$, $S_{j,s}$, $T_s$ is a polynomial in $s$ of degree at most $n$.
\item For all $n \in \Z_{\geq 0}$, the coefficient of $w^{\frac{n}{2}}$ in the universal power series $F_r$, $G_r$,  $A_r$, $A_{j,r}$, $B_r$, $B_{ij,r}$
is a polynomial in $r$ of degree at most $n$.
\end{enumerate}
\end{conjecture}
It is easy to see that the formulae for $V_s$, $W_s$, $X_s$, $Q_s$, $R_s$, $T_s$, $F_r$, $G_r$ of Conjectures \ref{conj1strong} and \ref{conj2strong} satisfy Conjecture \ref{polconj}.
As mentioned above, for $\rho=2$, we have computed the universal power series of Theorems \ref{structhm1} and \ref{structhm2}, up a certain order in $q$, for \emph{arbitrary} $s$ and $r$. Using these to compute the universal power series of Conjectures \ref{conj1strong} and \ref{conj2strong}, up to certain orders in $z$ and $w$, verifies Conjecture \ref{polconj} in these cases. 
For $\rho=3$ and $4$ we can use Conjecture \ref{polconj} to determine the universal power series for \emph{all} $s$ and $r$  up to certain orders in $z$ and $w$ by interpolation.\footnote{As mentioned above, we have no direct data for the rank 4 virtual Verlinde series. However, using the virtual Segre-Verlinde correspondence (Conjecture \ref{conj3}), one can obtain such data indirectly from the rank 4 virtual Segre series.} The coefficients of the power series are determined as solutions of overdetermined systems of linear equations; the existence of solutions gives further support for the conjecture. 

\section{Algebraicity} \label{sec:alg}

\subsection{Rank 1}

In this section, we give several conjectural expressions for the remaining power series in Conjecture \ref{conj1} for $\rho=1$. We conjecture
\begin{align*}
S_{-1}(z)&=\tfrac{1}{2}((1+4t)-(1+2t)^{\frac{1}{2}}(1+6t))^{\frac{1}{2}}),\\
S_2(z)&=0, \\
S_3(z)&=\tfrac{1}{2t}((1+t)(1-2t)-(1-t)(1-4t^2)^{\frac{1}{2}}),\\
Y_3(z)&=(\smfr{1}{2}+\smfr{1}{2}(1-4t^2)^{\frac{1}{2}})^{\frac{1}{2}},\\
Z_3(z)&=\tfrac{1}{2t^3}((1-t)^2(t+2t)-(1-t^2)(1-4t^2)^{\frac{1}{2}}),
\end{align*}
where $z = t (1+(1-s)t)^{1-s}$ and, in each case, $s$ is specialized to the value in the subscript. Moreover,  $S_4(z)$ conjecturally satisfies the following quartic equation 
$$
x^4 -8(1-3t)x^3 + \smfr{1}{t}(1-3t)(1+16t-60t^2)x^2 - \smfr{4}{t}(1-3t)^2(1-12t^2)x -16t(1-3t)^3=0.
$$
We verified these formulae up to order 35 in $t$ using the method described in Section \ref{sec:verif:rk1}.

\subsection{Rank 2} \label{sec:alg:rk2}
\addtocontents{toc}{\protect\setcounter{tocdepth}{2}}

In this section, we give conjectural expressions for some of the remaining power series in Conjectures \ref{conj1} and \ref{conj2} for $\rho=2$. For $\rho=2$, we conjecture in addition that the universal power series satisfy the following relations:
\begin{align*}
Y_{\{1\},s}(z^{\frac{1}{2}})&=Y_{\varnothing,s}(-z^{\frac{1}{2}}), \quad \ \ Z_{\{1\},s}(z^{\frac{1}{2}})=Z_{\varnothing,s}(-z^{\frac{1}{2}}), \quad S_{\{1\},s}(z^{\frac{1}{2}})=S_{\varnothing,s}(-z^{\frac{1} {2}}), \\
A_{\{1\},r}(w^{\frac{1}{2}})&=A_{\varnothing,r}(-w^{\frac{1}{2}}), \quad B_{\{1\},r}(w^{\frac{1}{2}})=B_{\varnothing,r}(-w^{\frac{1}{2}}),
\end{align*}
for all $s,r \in \Z$. We can therefore focus on the power series $Y_{s}:=Y_{\varnothing,s}$, $Z_{s}:= Z_{\varnothing,s}$, $S_s:=S_{\varnothing,s}$, $A_{r}:=A_{\varnothing,r}$, and $B_r:=B_{\varnothing,r}$.

\subsection*{Segre series}
\addtocontents{toc}{\protect\setcounter{tocdepth}{1}}

\subsection*{$\boldsymbol{s=0}$} For $z = t (1+t)$, we conjecture
 \begin{align*}
Y_{0}&=\frac{((1+t)^{\frac{1}{2}}+t^{\frac{1}{2}})(1+t)^2}{(1+2t)^\frac{1}{2}},\quad 
Z_0=1,\quad
S_0=t^{\frac{1}{2}}(1+t)^{\frac{1}{2}}.
\end{align*}

\subsection*{$\boldsymbol{s=1}$} For $z = t (1+\tfrac{1}{2}t)^{\frac{1}{2}}$, we conjecture
\begin{align*}
Y_1=(1+t)+t^{\frac{1}{2}}(1+\smfr{3}{4}t)^{\frac{1}{2}},\quad
Z_1=\frac{1+\smfr{3}{4}t -\tfrac{1}{2} t^{\frac{1}{2}}(1+\smfr{3}{4}t)^{\frac{1}{2}}}{1+\smfr{1}{2}t}, \quad S_1=-\smfr{1}{2}t+t^{\frac{1}{2}}(1+\smfr{3}{4}t)^{\frac{1}{2}}.
\end{align*}

\subsection*{$\boldsymbol{s=2}$} For $z = t$, we conjecture
\begin{align*}
Y_2&=t^{\frac{1}{2}} +(1+t)^{\frac{1}{2}},\quad
Z_2=1+t-t^{\frac{1}{2}}(1+t)^{\frac{1}{2}},\quad
S_2=-t+ t^{\frac{1}{2}}(1+t)^{\frac{1}{2}}.
\end{align*}

\subsection*{$\boldsymbol{s=3}$} For $z = t (1-\tfrac{1}{2}t)^{-\frac{1}{2}}$, we conjecture
\begin{align*}
Y_{3}&=1+t^{\frac{1}{2}}(1-\smfr{1}{4}t)^{\frac{1}{2}},\\
Z_3&=\frac{(1+\smfr{1}{2}t)((1-\smfr{1}{4}t)(1+\smfr{1}{2}t)-\smfr{3}{2}t^{\frac{1}{2}}(1-\smfr{1}{4}t)^{\frac{1}{2}}(1-\smfr{1}{6}t))}{(1-\smfr{1}{2}t)^3},\\
S_3&=\frac{-\smfr{3}{2}t(1-\smfr{1}{6}t)+t^{\frac{1}{2}} (1-\smfr{1}{4}t)^{\frac{1}{2}}(1+\smfr{1}{2}t)}{1-\frac{1}{2}t}.
\end{align*}

\subsection*{$\boldsymbol{s=4}$} For $z = t (1-t)^{-1}$, we conjecture
\begin{align*}
Y_4&=(1-t)^{\frac{1}{2}}+t^{\frac{1}{2}},\quad
Z_4=\frac{1-2t^{\frac{1}{2}}(1-t)^{\frac{1}{2}})}{(1-2t)^2},\quad
S_4=\frac{-2t(1-t)+t^{\frac{1}{2}}(1-t)^{\frac{1}{2}}}{1-2t}.
\end{align*}

\subsection*{$\boldsymbol{s=5}$} Consider the unique solutions $x$, $y$ of
\begin{align*}
x^4&-2(1-t)x^3+(1-\smfr{3}{2}t)^2x^2-2t^2(1-t)x+t^4=0 \\
y^4&-2(1+\smfr{3}{4}t)y^3+(1-\smfr{1}{2}t)(1+\smfr{3}{4}t)y^2-t(1-\smfr{11}{2}t)(1+\smfr{3}{4}t)^2=0
\end{align*}
having leading terms $x = 1+t^{\frac{1}{2}}+O(t)$ and $y = 1+\frac{3}{2}t^{\frac{1}{2}}+O(t)$ respectively. Then conjecturally we have
$$
x = Y_5, \quad y=\frac{Z_5Y_5^4}{(1-\smfr{1}{2}t)^3},
$$
where $z = t (1-\tfrac{3}{2}t)^{-\frac{3}{2}}$.


\subsection*{$\boldsymbol{s=-1}$} Consider the unique solutions $x$, $y$ of
\begin{align*}
t^4x^4& -2t^2(1+2t)x^3 + (1+\smfr{3}{2}t)^2x^2 -2(1+2t)x+1=0 \\
y^4& -2(1+\smfr{15}{4}t) y^3 + (1+\smfr{5}{2}t)(1+\smfr{15}{4}t)y^2 -t(1-\smfr{5}{2}t)(1+\smfr{15}{4}t)^2=0
\end{align*}
having leading terms $x=1+t^{\frac{1}{2}}+O(t)$ and $y=1-\smfr{3}{2}t^{\frac{1}{2}}+O(t)$ respectively. Then conjecturally we have
$$
x =\frac{Y_{-1}}{(1+\smfr{3}{2}t)^{2}}, \quad y=\frac{(1+\smfr{3}{2}t)^3Z_{-1}}{Y_{-1}^{2}},
$$
where $z = t (1+\tfrac{3}{2}t)^{\frac{3}{2}}$.

\subsection*{Verlinde series}

\subsection*{$\boldsymbol{r=2}$} For $w = v$, we conjecture
\begin{align*}
A_{2}=\frac{1+v^{\frac{1}{2}}}{1+v},\quad
B_{2}=\frac{1+v}{(1+v^{\frac{1}{2}})^2}.
\end{align*}

\subsection*{$\boldsymbol{r=1}$} For $w=v(1+v)^{-\frac{3}{4}}$, we conjecture
\begin{align*}
A_{1}&=\frac{1+\smfr{1}{2}v+v^{\frac{1}{2}}(1+\smfr{1}{4}v)^{\frac{1}{2}}}{1+v},\\
B_{1}&=(1+v)((1+v)(1+\smfr{1}{4}v)-\smfr{3}{2}v^{\frac{1}{2}}(1+\smfr{1}{3}v)(1+\smfr{1}{4}v)^{\frac{1}{2}}).
\end{align*}

\subsection*{$\boldsymbol{r=0}$} For $w=v(1+v)^{-1}$, we conjecture
\begin{align*}
A_{0}=1+\frac{v^{\frac{1}{2}}}{(1+v)^{\frac{1}{2}}},\quad
B_{0}=1+v-v^{\frac{1}{2}}(1+v)^{\frac{1}{2}}.
\end{align*}

\subsection*{$\boldsymbol{r=-1}$} For $w=v(1+v)^{-\frac{3}{4}}$, we conjecture
\begin{align*}
A_{-1}=1+\smfr{1}{2}v+v^{\frac{1}{2}}(1+\smfr{1}{4}v)^{\frac{1}{2}},\quad
B_{-1}=1+\smfr{1}{4}v-\tfrac{1}{2} v^{\frac{1}{2}} (1+\smfr{1}{4}v)^{\frac{1}{2}}.
\end{align*}

\subsection*{$\boldsymbol{r=-2}$} For $w = v$, we conjecture
\begin{align*}
A_{-2}=\frac{1}{1-v^\frac{1}{2}},\quad
B_{-2}=1.
\end{align*}

\subsection*{$\boldsymbol{r=-3,3}$} Consider the unique solutions $x_1, x_2$ of 
$$x^4 -(2+v)x^3+x^2-v^2(2+v)x+v^4=0$$
having leading terms $x_1=1+v^{\frac{1}{2}}+O(v)$ and $x_2=v^2(1+v^{\frac{1}{2}})+O(v^3)$. In addition, consider the unique solutions $y_1, y_2$ of 
$$y^4-2(1+\smfr{9}{4}v)y^3+(1+v)(1+\smfr{9}{4}v)y^2-v(1-4v)(1+\smfr{9}{4}v)^2=0$$
having leading terms  $y_1=1 + \smfr{3}{2}v^{\frac{1}{2}}+O(v)$ and $y_2=1 - \smfr{3}{2}v^{\frac{1}{2}}+O(v)$. Then conjecturally we have
\begin{align*}
x_1&=(1+v)^2A_{3}, \quad\quad \ \, x_2=\frac{v^2}{1+v}A_{-3}, \\
y_1&=(1+v)^{5}B_{3}A_{3}^4, \quad y_2=(1+v)^2\frac{B_{-3}}{A_{-3}^2},
\end{align*}
where $w=v(1+v)^{\frac{5}{4}}$. 





Using Theorems \ref{structhm1}, \ref{structhm2}, and the method described in Section \ref{sec:verif:higherrk}, we verified that the conjectural formulae of this subsection produce the correct virtual Segre and Verlinde numbers for the following surfaces up to the following virtual dimensions (always for certain values of $H,c_1$ such that there are no strictly semistable sheaves): 
\begin{itemize}
\item elliptic surface of type $E(3)$, blow-up of a $K3$ surface, the blow-up of either of the previous two surfaces in one point, elliptic surface of type $E(4)$; all up to virtual dimension $18$, 
\item double cover of $\PP^2$ branched along a smooth octic, its blow-up in one point, elliptic surface of type $E(5)$, double cover of $\PP^1\times \PP^1$ branched along a smooth curve of bidegree $(6,6)$; all up to virtual dimension $16$,
\item smooth quintic in $\PP^3$ up to virtual dimension $13$.
\end{itemize}

\subsection{Rank 3}  \label{sec:alg:rk3}
\addtocontents{toc}{\protect\setcounter{tocdepth}{2}}

Based on experimentation, we conjecture that
\begin{align*}
Y_{\{1,2\},s}(z^{\frac{1}{2}})&=Y_{\varnothing,s}(-z^{\frac{1}{2}}), \quad Y_{\{2\},s}=\tau(Y_{\{1\},s}),
\end{align*}
where $\tau$ denotes complex conjugation of the coefficients, and the same with $Y_{J,s}$ replaced by the corresponding $Z_{J,s}$, $S_{J,s}$, $A_{J,r}$, and $B_{J,r}$.
Therefore, we sometimes restrict attention to 
$Y_{\varnothing,s}$, $Z_{\varnothing,s}$, $S_{\varnothing,s}$, $A_{\varnothing,r}$, $B_{\varnothing,r}$ and $Y_{\{1\},s}$, $Z_{\{1\},s}$, $S_{\{1\},s}$, $A_{\{1\},r}$, $B_{\{1\},r}$.

\subsection*{Segre series}
\addtocontents{toc}{\protect\setcounter{tocdepth}{1}}

\subsection*{$\boldsymbol{s=0}$} For $z = t(1+t)$, we conjecturally have
\begin{align*}
S_{\varnothing,0}&=-3^{\frac{1}{2}}t^{\frac{1}{2}}(1+t)^{\frac{1}{2}}, \quad S_{\{1\},0}=0,\quad Z_{\varnothing,0}=2,  \quad Z_{\{1\},0}=1,\\
Y_{\varnothing,0}&=\frac{(1+t)^3}{(1+2t)^{\frac{3}{2}}+3^{\frac{1}{2}}t^{\frac{1}{2}}(1+t)^{\frac{1}{2}}(1+2t)^{\frac{1}{2}}},\quad
Y_{\{1\},0}=\frac{(1+t)^3}{(1+2t)^{\frac{1}{2}}(1-\varepsilon_3t)},
\end{align*}
where we recall that $\varepsilon_3 = \exp(2 \pi i / 3)$ with $i:=\sqrt{-1}$.

For $s=1,2,4,5$ we found explicit power series
$s_{1},\ldots, s_{4}$, $y_{1},\ldots, y_{4}$, $z_{1},\ldots, z_{4}$, with $s_2=0$,
such that we conjecturally have
\begin{align*}
S_{\varnothing,s}&=\tfrac{1}{2}\big(s_{1}+s_{2}-3^{\frac{1}{2}}t^{\frac{1}{2}}\big(\smfr{1}{3t}(s_{3}+s_{4})\big)^{\frac{1}{2}}\big),\
S_{\{1\},s}=\tfrac{1}{2}\big(s_{1}-s_{2}+3^{\frac{1}{2}}i t\big(-\smfr{1}{3t^2}(s_{3}-s_{4})\big)^{\frac{1}{2}}\big),\\
Y_{\varnothing,s}&=\tfrac{1}{2}\big(y_{1}+y_{2}-3^{\frac{1}{2}}t^{\frac{1}{2}}\big(\smfr{1}{3t}(y_{3}+y_{4})\big)^{\frac{1}{2}}\big),\
Y_{\{1\},s}=\tfrac{1}{2}\big(y_{1}-y_{2}+3^{\frac{1}{2}}i t\big(-\smfr{1}{3t^2}(y_{3}-y_{4})\big)^{\frac{1}{2}}\big),\\
Z_{\varnothing,s}&=\tfrac{1}{2}\big(z_{1}+z_{2}+3^{\frac{1}{2}}t^{\frac{1}{2}}\big(\smfr{1}{3t}(z_{3}+z_{4})\big)^{\frac{1}{2}}\big),\
Z_{\{1\},s}=\tfrac{1}{2}\big(z_{1}-z_{2}-3^{\frac{1}{2}}i t\big(-\smfr{1}{3t^2}(z_{3}-z_{4})\big)^{\frac{1}{2}}\big).
\end{align*}
Here we choose roots as follows: we write the term inside the brackets of $(\bullet)^{\frac{1}{2}}$ as a power series in $t$ starting with $a^2$, with $a\in \Q_{>0}$, then $(\bullet)^{\frac{1}{2}}$ is
a power series in $t$ starting with $a$. Below, we list $s_{1},\ldots, s_{4}$, $y_{1},\ldots, y_{4}$, $z_{1},\ldots, z_{4}$ for $s=1,2$. 

\subsection*{$\boldsymbol{s=1}$} In this case $z$ and $t$ are related by $z = t(1+\tfrac{2}{3}t)^{\frac{2}{3}}$. Conjecturally, the power series
$S_{\varnothing,1}$, $S_{\{1,2\},1}$, $S_{\{1\},1}$, $S_{\{2\},1}$ are the four solutions of 
\begin{align*}
x^4 + 2tx^3 - (3t+ t^2)x^2 - (3t^2 + 2t^3)x -(t^3 + \smfr{2}{3}t^4)=0.
\end{align*}
Conjecturally, $Y_{\varnothing,1}$, $Y_{\{1,2\},1}$, $Y_{\{1\},1}$, $Y_{\{2\},1}$ are the four solutions of 
\begin{align*}
x^4-(4+\smfr{17}{3}t)(1+\smfr{2}{3}t)^{\frac{1}{2}}x^3+(6 + 18t + 16t^2 + \smfr{31}{9}t^3)x^2-(4+\smfr{17}{3}t)(1+\smfr{2}{3}t)^{\frac{7}{2}}x+(1+\smfr{2}{3}t)^6=0.
\end{align*}
Conjecturally,  $Z_{\varnothing,1}$, $Z_{\{1,2\},1}$, $Z_{\{1\},1}$, $Z_{\{2\},1}$ are the four solutions of 
\begin{align*}
x^4-6\frac{1+\smfr{10}{9}t}{1+\smfr{2}{3}t}x^3+\frac{(13+\smfr{58}{3}t+\smfr{55}{9}t^2)(1+\smfr{10}{9}t)}{(1+\smfr{2}{3}t)^3}x^2+\frac{(4+\smfr{5}{3}t)(1+\smfr{10}{9}t)^2}{(1+\smfr{2}{3}t)^3}(-3x+1)=0.
\end{align*}
Explicitly, using the notation introduced above, this can be written as
\begin{align*}
s_1&=-t, \quad s_3=6t(1+\smfr{5}{6}t), \quad s_4=6t(1+\smfr{2}{3}t)^{\frac{1}{2}}(1+\smfr{10}{9}t)^{\frac{1}{2}},\\
y_1&=(2+\smfr{17}{6}t)(1+\smfr{2}{3}t)^{\frac{1}{2}},\quad
y_2=\smfr{3}{2}t(1+\smfr{10}{9}t)^{\frac{1}{2}},\\
y_3&=6t+\smfr{25}{2}t^2+\smfr{20}{3}t^3,\quad
y_4=(6t+\smfr{17}{2}t^2)(1+\smfr{2}{3}t)^{\frac{1}{2}}(1+\smfr{10}{9}t)^{\frac{1}{2}}, \\
z_1&=\frac{3+\smfr{10}{3}t}{1+\smfr{2}{3}t}, \quad
z_2=\frac{(1+\smfr{5}{3}t)(1+\smfr{10}{9}t)^{\frac{1}{2}}}{(1+\smfr{2}{3}t)^{\frac{3}{2}}},\quad
z_3=\frac{6t+\smfr{35}{3}t^2+\smfr{50}{9}t^3}{(1+\smfr{2}{3}t)^3},\quad 
z_4=\frac{6t(1+\smfr{10}{9}t)^{\frac{3}{2}}}{(1+\smfr{2}{3}t)^{\frac{5}{2}}}.
\end{align*}


\begin{remark}
We briefly sketch the method we use to find these power series, and those for $s=2,4,5$ below. The same method is used to find the power series for the virtual Verlinde numbers for $r=-2,-1,1,2$ below. 
Let $L_{J,s}$ be any of the power series $Y_{J,s}$, $Z_{J,s}$, $S_{J,s}$. 
Then we expect that the four power series
\begin{align*}
&L_{\varnothing,s}+L_{\{1,2\},s}+L_{\{1\},s}+L_{\{2\},s},\ (L_{\varnothing,s}+L_{\{1,2\},s})(L_{\{1\},s}+L_{\{2\},s}),\\
&L_{\varnothing,s}L_{\{1,2\},s}+L_{\{1\},s}L_{\{2\},s},\ L_{\varnothing,s}L_{\{1,2\},s}L_{\{1\},s}L_{\{2\},s}
\end{align*}
 are simple algebraic functions, for which we can guess a formula from their coefficients modulo $t^7$. Moreover, by Conjectures \ref{conj1strong} and \ref{conj2strong} (and the discussion at the end of Section \ref{sec:generalconj}), we also have 
 $$
 Y_{\varnothing,s}Y_{\{1,2\},s}=Y_{\{1\},s}Y_{\{2\},s}, \quad S_{\varnothing,s}+S_{\{1,2\},s}=S_{\{1\},s}+S_{\{2\},s}.
 $$ 
Then the explicit expressions for $L_{\varnothing,s}$, $L_{\{1,2\},s}$, $L_{\{1\},s}$, $L_{\{2\},s}$ are found from those of the above four series by double extraction of square roots. The algebraic equations for $s=1,2,4,5$ and $r=-2,-1,1,2$ are then just obtained as the product 
$$(x-L_{\varnothing,s})(x-L_{\{1,2\},s})(x-L_{\{1\},s})(x-L_{\{2\},s}).$$
Indeed we find modulo $t^7$ \begin{align*}
S_{\varnothing,1}+S_{\{1,2\},1}+S_{\{1\},1}+S_{\{2\},1}&=-2t, \quad (S_{\varnothing,1}+S_{\{1,2\},1})(S_{\{1\},1}+S_{\{2\},1})=t^2,\\
S_{\varnothing,1}S_{\{1,2\},1}+S_{\{1\},1}S_{\{2\},1}&=-3t(1+\smfr{2}{3}t),\quad S_{\varnothing,1}S_{\{1,2\},1}S_{\{1\},1}S_{\{2\},1}=-t^3(1+\smfr{2}{3}t), \\
Y_{\varnothing,1}+Y_{\{1,2\},1}+Y_{\{1\},1}+Y_{\{2\},1}&=(4+\smfr{17}{3}t)(1+\smfr{2}{3}t)^{\frac{1}{2}},\\(Y_{\varnothing,1}+Y_{\{1,2\},1})(Y_{\{1\},1}+Y_{\{2\},1})&=4+14t + \smfr{40}{3}t^2 + \smfr{77}{27}t^3,\\
Y_{\varnothing,1}Y_{\{1,2\},1}=Y_{\{1\},1}Y_{\{2\},1}&=(1+\smfr{2}{3}t)^3, \\
Z_{\varnothing,1}+Z_{\{1,2\},1}+Z_{\{1\},1}+Z_{\{2\},1}&=\frac{6+\smfr{20}{3}t}{1+\smfr{2}{3}t}, \\
(Z_{\varnothing,1}+Z_{\{1,2\},1})(Z_{\{1\},1}+Z_{\{2\},1})&=\frac{(2+\smfr{7}{3}t)(4+\smfr{5}{3}t)(1+\smfr{10}{9}t)}{(1+\smfr{2}{3}t)^3},\\
Z_{\varnothing,1}Z_{\{1,2\},1}+Z_{\{1\},1}Z_{\{2\},1}&=\frac{5(1+\smfr{10}{9}t)}{1+\smfr{2}{3}t},\quad
Z_{\varnothing,1}Z_{\{1,2\},1}Z_{\{1\},1}Z_{\{2\},1}=\frac{(4+\smfr{5}{3}t)(1+\smfr{10}{9}t)^2}{(1+\smfr{2}{3}t)^3}.
\end{align*}
\end{remark}

\subsection*{$\boldsymbol{s=2}$} In this case $z$ and $t$ are related by $z = t(1+\tfrac{1}{3}t)^{\frac{1}{3}}$. Conjecturally, the power series
$S_{\varnothing,2}$, $S_{\{1,2\},2}$, $S_{\{1\},2}$, $S_{\{2\},2}$ are the four solutions of 
\begin{align*}
x^4 + 4tx^3 + (-3t+ 4t^2)x^2 - 6t^2x -4t^3 - \smfr{4}{3}t^4=0.
\end{align*}
Conjecturally, $Y_{\varnothing,2}$, $Y_{\{1,2\},2}$, $Y_{\{1\},2}$, $Y_{\{2\},2}$ are the four solutions of 
\begin{align*}
x^4-(4+\smfr{7}{3}t)(1+\smfr{4}{3}t)^{\frac{1}{2}}x^3+(6+8t)(1+\smfr{1}{3}t)^2 x^2-(4+\smfr{7}{3}t)(1+\smfr{4}{3}t)^{\frac{1}{2}}(1+\smfr{1}{3}t)^3x+(1+\smfr{1}{3}t)^6=0.
\end{align*}
Conjecturally,  $Z_{\varnothing,2}$, $Z_{\{1,2\},2}$, $Z_{\{1\},2}$, $Z_{\{2\},2}$ are the four solutions of 
\begin{align*}
x^4-\frac{(6+5t)(1+\smfr{4}{3}t)(1+\smfr{4}{9}t)}{(1+\smfr{1}{3}t)^3}x^3+\frac{13(1+\smfr{4}{3}t)^2(1+\smfr{4}{9}t)}{(1+\smfr{1}{3}t)^3}x^2+\frac{4(1+\smfr{4}{3}t)^2(1+\smfr{4}{9}t)^2}{(1+\smfr{1}{3}t)^4}(-3x+1)=0.
\end{align*}
Explicitly, using the notation introduced above, this can be written as
\begin{align*}
s_1&=-2t, \quad s_3=6t+4t^2, \quad s_4=6t(1+\smfr{4}{3}t)^{\frac{1}{2}}(1+\smfr{4}{9}t)^{\frac{1}{2}},\\
y_1&=(2+\smfr{7}{6}t)(1+\smfr{4}{3}t)^{\frac{1}{2}},\quad
y_2=\smfr{3}{2}t(1+\smfr{4}{9}t)^{\frac{1}{2}},\\
y_3&=6t+\smfr{17}{2}t^2+\smfr{8}{3}t^3,\quad
y_4=(6t+\smfr{7}{2}t^2)(1+\smfr{4}{3}t)^{\frac{1}{2}}(1+\smfr{4}{9}t)^{\frac{1}{2}},\\
z_1&=\frac{3(1+\smfr{5}{6}t)(1+\smfr{4}{9}t)(1+\smfr{4}{3}t)}{(1+\smfr{1}{3}t)^3}, \\
z_2&=\frac{(1+\smfr{25}{6}t+\smfr{16}{9}t^2)(1+\smfr{4}{3}t)^{\frac{1}{2}}(1+\smfr{4}{9}t)^{\frac{1}{2}}}{(1+\smfr{1}{3}t)^3},\\
z_3&=\frac{(24t + \smfr{101}{2}t^2 + \smfr{98}{3}t^3 + \smfr{20}{3}t^4)(1+\smfr{4}{3}t)(1+\smfr{4}{9}t)}{(1+\smfr{1}{3}t)^6},\\
z_4&=\frac{(24t+  \smfr{59}{2}t^2 + \smfr{26}{3}t^3)(1+\smfr{4}{3}t)^{\frac{3}{2}}(1+\smfr{4}{9}t)^{\frac{3}{2}}}{(1+\smfr{1}{3}t)^6}.
\end{align*}

\subsection*{$\boldsymbol{s=3}$} For $z = t$, we conjecturally have
\begin{align*}
S_{\varnothing,3}&=-\smfr{3}{2}t-3^{\frac{1}{2}}t^\frac{1}{2}(1+\smfr{3}{4}t)^{\frac{1}{2}},\quad 
S_{\{1\},3}=-\tfrac{1}{2}(3+3^{\frac{1}{2}}i)t,\\
Y_{\varnothing,3}&=1+\smfr{3}{2}t-3^{\frac{1}{2}}t^\frac{1}{2}(1+\smfr{3}{4}t)^{\frac{1}{2}},\quad 
Y_{\{1\},3}=1, \\
Z_{\varnothing,3}&=\frac{2(1+t)(1+\frac{3}{4}t)^{\frac{1}{2}}}{(1+3t)(1+\frac{3}{4}t)^{\frac{1}{2}}-\frac{3}{2}3^{\frac{1}{2}}t^{\frac{1}{2}}(1+t)},\quad
Z_{\{1\},3}=1+t.
\end{align*}

\subsection*{$\boldsymbol{s=4}$} In this case $z$ and $t$ are related by $z = t(1-\tfrac{1}{3}t)^{-\frac{1}{3}}$. Conjecturally, the power series
$S_{\varnothing,4}$, $S_{\{1,2\},4}$, $S_{\{1\},4}$, $S_{\{2\},4}$ are the four solutions of 
\begin{align*}
x^4 + 2tx^3 - (3t + t^2)x^2 - (3t^2 + 2t^3)x - (t^3 + \smfr{2}{3}t^4)=0.
\end{align*}
Conjecturally, $Y_{\varnothing,4}$, $Y_{\{1,2\},4}$, $Y_{\{1\},4}$, $Y_{\{2\},4}$ are the four solutions of 
\begin{align*}
x^4-(4-\smfr{1}{3}t)(1+\smfr{2}{3}t)^{\frac{1}{2}}x^3+(6+4t)(1-\smfr{1}{3}t)^2x^2-(4-\smfr{1}{3}t)(1-\smfr{1}{3}t)^3(1+\smfr{2}{3}t)^{\frac{1}{2}}x+(1-\smfr{1}{3}t)^6=0.
\end{align*}
Conjecturally,  $Z_{\varnothing,4}$, $Z_{\{1,2\},4}$, $Z_{\{1\},4}$, $Z_{\{2\},4}$ are the four solutions of 
\begin{align*}
x^4&-\frac{(6 + 14t - \smfr{7}{3}t^2)(1-\smfr{2}{9}t)(1+\smfr{2}{3}t)^3}{(1-\smfr{1}{3}t)^6}x^3+\frac{(13 +36t + \smfr{23}{3}t^2 - \smfr{298}{27}t^3)(1-\smfr{2}{9}t)(1+\smfr{2}{3}t)^4}{(1-\smfr{1}{3}t)^8}x^2\\&+\frac{4(1-\smfr{2}{9}t)^2(1+\smfr{2}{3}t)^8}{(1-\smfr{1}{3}t)^{10}}(-3x+1)=0.
\end{align*}

\subsection*{$\boldsymbol{s=5}$} In this case $z$ and $t$ are related by $z = t(1-\tfrac{2}{3}t)^{-\frac{2}{3}}$. Conjecturally, the power series
$S_{\varnothing,5}$, $S_{\{1,2\},5}$, $S_{\{1\},5}$, $S_{\{2\},5}$ are the four solutions of 
\begin{align*}
x^4 &+\frac{2t(5 - \smfr{20}{3}t + \smfr{11}{9}t^2)}{(1-\smfr{2}{3}t)^2}x^3+  \frac{-3t + 45t^2 - \smfr{305}{3}t^3 + \smfr{740}{9}t^4 - \smfr{220}{9}t^5 + \smfr{193}{81}t^6}{(1-\smfr{2}{3}t)^4}x^2\\&-\frac{3t^2(5 - \smfr{20}{3}t + \smfr{11}{9}t^2)(1 - \smfr{14}{3}t + t^2)}{(1-\smfr{2}{3}t)^3}x-\frac{t^3(5 - \smfr{20}{3}t + \smfr{11}{9}t^2)^2}{(1-\smfr{2}{3}t)^3}=0.
\end{align*}
Conjecturally, $Y_{\varnothing,5}$, $Y_{\{1,2\},5}$, $Y_{\{1\},5}$, $Y_{\{2\},5}$ are the four solutions of 
\begin{align*}
x^4-(4+\smfr{1}{3}t)(1-\smfr{2}{3}t)^{\frac{1}{2}}x^3+(6-6t-\smfr{1}{9}t^3)x^2-(4+\smfr{1}{3}t)(1-\smfr{2}{3}t)^{\frac{7}{2}}x+(1-\smfr{2}{3}t)^6=0.
\end{align*}
Conjecturally,  $Z_{\varnothing,5}$, $Z_{\{1,2\},5}$, $Z_{\{1\},5}$, $Z_{\{2\},5}$ are the four solutions of 
\begin{align*}
x^4&-\frac{(6 + 22t - \smfr{7}{3}t^2)(1-\smfr{2}{9}t)(1+\smfr{1}{3}t)^4}{(1-\smfr{2}{3}t)^7}x^3\\&+\frac{(13 + \smfr{14}{3}t - \smfr{851}{9}t^2 + \smfr{1943}{27}t^3 + \smfr{5455}{81}t^4 - \smfr{19009}{243}t^5 +\smfr{11623}{729}t^6 - \smfr{202}{2187}t^7 - \smfr{689}{6561}t^8)(1-\smfr{2}{9}t)(1+\smfr{1}{3}t)^6}{(1-\smfr{2}{3}t)^{15}}x^2\\&+\frac{4(1-\smfr{2}{9}t)^2(1+\smfr{1}{3}t)^{12}(1-\smfr{11}{12}t)}{(1-\smfr{2}{3}t)^{15}}(-3x+1)=0.
\end{align*}

\subsection*{$\boldsymbol{s=6}$} For $z = t(1-t)^{-1}$, we conjecturally have
\begin{align*}
S_{\varnothing,6}&=\frac{1-2t}{1-3^{-\frac{1}{2}}t^{-\frac{1}{2}}(1-t)^{-\frac{1}{2}}},\quad
S_{\{1\},6}=\frac{4t(1-t)}{2t-1-3^{-\frac{1}{2}}i},\\
Y_{\varnothing,6}&=1-3^{\frac{1}{2}}t^{\frac{1}{2}}(1-t)^{\frac{1}{2}},\quad 
Y_{\{1\},6}=1-\smfr{1}{2}(3+3^{\frac{1}{2}}i)t,\\
Z_{\varnothing,6}&=\frac{2}{\big(1-3^{\frac{1}{2}}t^{\frac{1}{2}}(1-t)^{\frac{1}{2}}\big)^3},\quad
Z_{\{1\},6}=\frac{1}{\big(1-\smfr{1}{2}(3+3^{\frac{1}{2}}i)t\big)^{3}}.
\end{align*}

\subsection*{Verlinde series} 

\subsection*{$\boldsymbol{r=0}$} For $w = v(1+v)^{-1}$, we conjecturally have
\begin{align*}
A_{\varnothing,0}&=\frac{(1+\smfr{3}{2}v)-3^{\frac{1}{2}}v^{\frac{1}{2}}(1+\smfr{3}{4}v)^{\frac{1}{2}}}{1+v},\quad
A_{\{1\},0}=(1+v)^{-1},\\
B_{\varnothing,0}&=2(1+v)(1+3v)(1+\smfr{3}{4}v)+3^{\frac{3}{2}}v^{\frac{1}{2}}(1+v)^2(1+\smfr{3}{4}v)^{\frac{1}{2}},\quad
B_{\{1\},0}=(1+v).
\end{align*}

\subsection*{$\boldsymbol{r=-1,1}$} In this case $w$ and $v$ are related by $w = v(1+v)^{-\frac{8}{9}}$. 
Conjecturally, $A_{\varnothing,-1}$, $A_{\{1,2\},-1}$, $A_{\{1\},-1}$, $A_{\{2\},-1}$ are the four solutions of 
\begin{align*}
x^4 -(4+v)x^3 + 6x^2 -\frac{4+v}{1+v}x + \frac{1}{(1+v)^2}=0.
\end{align*}
Conjecturally,  $B_{\varnothing,-1}$, $B_{\{1,2\},-1}$, $B_{\{1\},-1}$, $B_{\{2\},-1}$ are the four solutions of
\begin{align*}
x^4 - 6(1+v)(1+\smfr{1}{2}v)(1+\smfr{1}{9}v)x^3 + 13(1+v)^2(1+\smfr{1}{9}v)x^2 + (1+v)^2(1+\smfr{1}{9}v)^2(-12x+4)=0.
\end{align*}

\noindent Similarly, $A_{\varnothing,1}$, $A_{\{1,2\},1}$, $A_{\{1\},1}$, $A_{\{2\},1}$ are the four solutions of 
\begin{align*}
x^4-\frac{4+v}{1+v}x^3+\frac{6}{(1+v)^2}x^2-\frac{4+v}{(1+v)^4}x+\frac{1}{(1+v)^6}=0
\end{align*}
and $B_{\varnothing,1}$, $B_{\{1,2\},1}$, $B_{\{1\},1}$, $B_{\{2\},1}$ are the four solutions of
\begin{align*}
x^4&-(6+18v+3v^2)(1+v)^3(1+\smfr{1}{9}v)x^3+(13 + 49v + 36v^2 - 4v^3 )(1+v)^4(1+\smfr{1}{9}v)x^2\\&+(1+v)^8(1+\smfr{1}{9}v)^2(-12x+4)=0.
\end{align*}

\subsection*{$\boldsymbol{r=-2,2}$} In this case $w$ and $v$ are related by $w = v(1+v)^{-\frac{5}{9}}$. 
Conjecturally, $A_{\varnothing,-2}$, $A_{\{1,2\},-2}$, $A_{\{1\},-2}$, $A_{\{2\},-2}$ are the four solutions of 
\begin{align*}
x^4 -(4 + 3v)x^3 + (6 + 6v -v^3)x^2 - (4 + 3v)x +1=0.
\end{align*}
Conjecturally,  $B_{\varnothing,-2}$, $B_{\{1,2\},-2}$, $B_{\{1\},-2}$, $B_{\{2\},-2}$ are the four solutions of 
\begin{align*}
x^4 - 6(1 + \smfr{4}{9}v)x^3 + (13 +2v-v^2)(1 + \smfr{4}{9}v)x^2 +(1-\smfr{1}{4}v)(1 + \smfr{4}{9}v)^2(-12x+4)=0.
\end{align*}

\noindent Conjecturally, $A_{\varnothing,2}$, $A_{\{1,2\},2}$, $A_{\{1\},2}$, $A_{\{2\},2}$ are the four solutions of 
\begin{align*}
x^4 -\frac{4 + 3v}{(1+v)^2}x^3 + \frac{6 + 6v -v^3}{(1+v)^4}x^2 - \frac{4 + 3v}{(1+v)^6}x +\frac{1}{(1+v)^8}=0.
\end{align*}
Conjecturally,  $B_{\varnothing,2}$, $B_{\{1,2\},2}$, $B_{\{1\},2}$, $B_{\{2\},2}$ are the four solutions of 
\begin{align*}
x^4&-6(1+v)^4(1+\smfr{4}{9}v)(1 + 5v + \smfr{5}{2}v^2)x^3+(1+v)^6(1+\smfr{4}{9}v)(13 + 74v + 89v^2 - 47v^3 - 95v^4 \\&- 11v^5 + 17v^6 + 2v^7 - v^8)x^2+(1-\smfr{1}{4}v)(1+v)^{12}(1+\smfr{4}{9}v)^2(-12x+4)=0.
\end{align*}

\subsection*{$\boldsymbol{r=-3,3}$} In this case $w=v$. Recall $\varepsilon_3 := \exp(2 \pi i /3) = -\smfr{1}{2}+\smfr{3^{\frac{1}{2}} i}{2}$. 
 Then we conjecturally have
\begin{align*}
A_{\varnothing,-3}&=\frac{1}{1+3^{\frac{1}{2}}v^{\frac{1}{2}}+v},\quad 
A_{\{1\},-3}=\frac{1}{1+\varepsilon_3^{-1}v},\quad 
B_{\varnothing,-3}=2,\quad B_{{\{1\}},-3}=1,\\
A_{\varnothing,3}&=\frac{1-3^{\frac{1}{2}}v^{\frac{1}{2}}+v}{(1+v)^2},\
A_{\{1\},3}=\frac{1+\varepsilon_3^{-1} v}{(1+v)^2},\
B_{\varnothing,3}=\frac{2(1+v)^3}{(1-3^{\frac{1}{2}}v^{\frac{1}{2}}+v)^3},\ 
B_{\{1\},3}=\frac{(1+v)^3}{(1+\varepsilon_3^{-1}v)^3}.
\end{align*}

Using $A_{\varnothing,r}A_{\{1,2\},r}=A_{\{1\},r}A_{\{2\},r}$, these expressions are related by 
the two different factorizations
$$\frac{1+v^3}{1+v}=(1+\varepsilon_3 v)(1+\varepsilon_3^{-1}v)=(1+3^{\frac{1}{2}}v^{\frac{1}{2}}+v)(1-3^{\frac{1}{2}}v^{\frac{1}{2}}+v).$$

Using Theorems \ref{structhm1}, \ref{structhm2}, and the method described in Section \ref{sec:verif:higherrk}, we verified that the conjectural formulae of this subsection produce the correct virtual Segre and Verlinde numbers for the following surfaces up to the following virtual dimensions (always for certain values of $H,c_1$ such that there are no strictly semistable sheaves): 
\begin{itemize}
\item \textbf{Segre case.} For $S$ a $K3$ surface up to virtual dimension $14$, for $S$ the blow-up of a $K3$ surface up to virtual dimension $14$, for $S$ an elliptic surface of type $E(3)$ up to virtual dimension $12$, for $S$ a double cover of $\PP^2$ branched along a smooth octic up to virtual dimension $6$, and for the blow-ups of these surfaces in one point with the same dimension bounds.
\item \textbf{Verlinde case.} For $S$ a $K3$ up to virtual dimension $12$, for $S$ the blow-up of a $K3$ surface up to virtual dimension $12$, for $S$ an elliptic surface of type $E(3)$ up to virtual dimension $8$, for $S$ a double cover of $\PP^2$ branched along a smooth octic up to virtual dimension $6$, for blow-ups of these surfaces in one point with the same dimension bounds 
\end{itemize}

\subsection{Rank 4} \label{sec:alg:rk4}
\addtocontents{toc}{\protect\setcounter{tocdepth}{2}}

We also have some partial results in the case of rank $\rho=4$.  We computed the universal power series modulo $t^7$ for the virtual Segre series. Recall that for $\rho=4$, we have no direct data for the rank 4 virtual Verlinde series, so we \emph{assume} the virtual Segre-Verlinde correspondence (Conjecture \ref{conj3}) in order to obtain such data from the rank 4 virtual Segre series. This determines the virtual Verlinde series modulo $v^5$. 

Let $L_{J,s}$ be any of the power series $Y_{J,s}$, $Z_{J,s}$, or $S_{J,s}$. Based on experimentation, we conjecture the following: 
\begin{align*}
L_{J,s} &\in \Q(i,2^{\frac{1}{2}})[[z^{\frac{1}{2}}]], \quad \forall J \subset [3], s \in \Z \\
L_{\{1,2,3\},s}(z^{\frac{1}{2}})&=L_{\varnothing,s}(-z^{\frac{1}{2}}),\quad  L_{\{1,3\},s}=\sigma(L_{\varnothing,s}), \quad L_{\{2\},s}(z^{\frac{1}{2}})=\sigma(L_{\varnothing,s})(-z^{\frac{1}{2}}),\\
L_{\{1,2\},s}(z^{\frac{1}{2}})&=L_{\{1\},s}(-z^{\frac{1}{2}}),\quad  L_{\{3\},s}=\tau(L_{\{1\},s}), \quad  L_{\{2,3\},s}(z^{\frac{1}{2}})=\tau(L_{\{1\},s})(-z^{\frac{1}{2}}).
\end{align*}
Here $\sigma$ is the involution of  $\Q(i,2^{\frac{1}{2}})[[z^{\frac{1}{2}}]]$, that replaces $2^{\frac{1}{2}}$ by $-2^{\frac{1}{2}}$ in the coefficients of the power series, and $\tau$ is complex conjugation of the coefficients. We expect the analogs of these statements to hold on the Verlinde side as well.
Therefore, in what follows, we restrict attention to $Y_{\varnothing,s}$, $Z_{\varnothing,s}$, $S_{\varnothing,s}$, $A_{\varnothing,r}$, $B_{\varnothing,r}$ and $Y_{\{1\},s}$, $Z_{\{1\},s}$, $S_{\{1\},s}$, $A_{\{1\},r}$, $B_{\{1\},r}$.

\subsection*{Segre series}
\addtocontents{toc}{\protect\setcounter{tocdepth}{1}}


\subsubsection*{$\boldsymbol{s=0}$} 

For $z=t(1+t)$, we conjecturally have 
\begin{align*}
S_{\varnothing,0}&=(1+2^{\frac{1}{2}})t^{\frac{1}{2}}(1+t)^{\frac{1}{2}},\quad S_{\{1\},0}=t^{\frac{1}{2}}(1+t)^{\frac{1}{2}},\quad Z_{\varnothing,0}=2(2+2^{\frac{1}{2}}),\quad Z_{\{1\},0}=2,\\
Y_{\varnothing,0}&=\frac{(1+t)^4((1+t)^{\frac{1}{2}}+t^{\frac{1}{2}})}{(1+2t)^{\frac{1}{2}}\big((1+2t)-2^{\frac{1}{2}}t^{\frac{1}{2}}(1+t)^{\frac{1}{2}}\big)},\quad
Y_{\{1\},0}=\frac{(1+t)^4((1+t)^{\frac{1}{2}}+t^{\frac{1}{2}})}{(1+2t)^{\frac{1}{2}}(1+(1-i)t)}.
\end{align*}

\subsubsection*{$\boldsymbol{s=4}$} 
We write 
$$z_1=((1+2t+2t^2)+2t^{\frac{3}{2}}(1+t)^{\frac{1}{2}})(1+\smfr{1}{2}t)^{\frac{1}{2}},\quad z_2=-\smfr{1}{2}(1-2t-2t^2)(1+t)^{\frac{1}{2}}+t^{\frac{1}{2}}(1+t)^2.$$
For $z=t$, we conjecture the following formulae to hold
\begin{align*}
S_{\varnothing,4}&=-2t+t^{\frac{1}{2}}\big((1+t)^{\frac{1}{2}}+2^{\frac{1}{2}}(1+\smfr{1}{2}t)^{\frac{1}{2}}\big),\quad
S_{\{1\},4}=(-2-i)t+t^{\frac{1}{2}}(1+t)^{\frac{1}{2}}, \\
Z_{\varnothing,4}&=\frac{2(1+t)^2(1+\smfr{1}{2}t)^{\frac{1}{2}}}{z_1+2^{\frac{1}{2}}z_2},\quad
Z_{\{1\},4}=2(1+t)^2\big((1+2t)-2t^{\frac{1}{2}}(1+t)^{\frac{1}{2}}\big),\\
Y_{\varnothing,4}&=((1+t)^{\frac{1}{2}}+t^{\frac{1}{2}})((1+t)+2^{\frac{1}{2}} t^{\frac{1}{2}} (1+\smfr{1}{2}t)^{\frac{1}{2}}),\quad
Y_{\{1\},4}=(1+t)^{\frac{1}{2}}+t^{\frac{1}{2}}.
\end{align*}



\subsection*{Verlinde series} 
\subsubsection*{$\boldsymbol{r=-4,4}$} 
In this case $w=v$. Conjecturally, we have the following attractive formulae
\begin{align*}
A_{\varnothing,-4}&=\frac{1}{(1-v^{\frac{1}{2}})(1-2^{\frac{1}{2}}v^{\frac{1}{2}}+v)},\quad 
A_{\{1\},-4}=\frac{1}{(1-v^{\frac{1}{2}})(1-i v)},\\
B_{\varnothing,-4}&=2(2+2^{\frac{1}{2}}),\quad B_{\{1\},-4}=2,\\
A_{\varnothing,4}&=\frac{(1+v^{\frac{1}{2}})(1+2^{\frac{1}{2}}v^{\frac{1}{2}}+v)}{(1+v)^3},\quad 
A_{\{1\},4}=\frac{(1+v^{\frac{1}{2}})(1-i v)}{(1+v)^3},\\
B_{\varnothing,4}&=2(2+2^{\frac{1}{2}})\frac{(1+v)^6}{(1+v^{\frac{1}{2}})^3(1+2^{\frac{1}{2}}v^{\frac{1}{2}}+v)^3},\quad 
B_{\{1\},4}=2\frac{(1+v)^6}{(1+v^{\frac{1}{2}})^3(1-i v)^3}.
\end{align*}
Using $A_{\varnothing,r}A_{\{1,2,3\},r}=A_{\{1\},r}A_{\{2,3\},r}$, these expressions are related by 
the two different factorizations
$$
\frac{1-v^4}{1+v}=(1-v^{\frac{1}{2}})(1+v^{\frac{1}{2}})(1+i v)(1-iv)=(1-v^{\frac{1}{2}})(1+v^{\frac{1}{2}})(1+2^{\frac{1}{2}}v^{\frac{1}{2}}+v)(1-2^{\frac{1}{2}}v^{\frac{1}{2}}+v).
$$

Using Theorem \ref{structhm1}, and the method described in Section \ref{sec:verif:higherrk}, we verified that the conjectural formulae of this subsection produce the correct virtual Segre numbers for the following surfaces up to the following virtual dimensions (always for certain values of $H,c_1$ such that there are no strictly semistable sheaves): 
\begin{itemize}
\item \textbf{Segre case.} For $S$ a $K3$ surface up to virtual dimension $6$, for $S$ the blow-up of a $K3$ surface in a point up to virtual dimension $6$, for $S$ a ``virtual surface'' with $K^2=-1$, $\chi=0$ up to virtual dimension $9$, and for $S$ a ``virtual surface'' with $K^2=-1$, $\chi=1$ up to virtual dimension $8$.
\end{itemize}
Assuming the virtual Segre-Verlinde correspondence (Conjecture \ref{conj3}), the conjectural formulae of this subsection produce the correct virtual Verlinde numbers for the following surfaces up to the following virtual dimensions (always for certain values of $H,c_1$ such that there are no strictly semistable sheaves): 
\begin{itemize}
\item \textbf{Verlinde case.} For $S$ a $K3$ surface up to virtual dimension $5$, for $S$ the blow-up of a $K3$ surface in a point up to virtual dimension $5$, for $S$ a ``virtual surface'' with $K^2=-1$, $\chi=0$ up to virtual dimension $8$, and for $S$ a ``virtual surface'' with $K^2=-1$, $\chi=1$ up to virtual dimension $6$. 
\end{itemize}

\subsection{Galois actions} \label{sec:Gal}
\addtocontents{toc}{\protect\setcounter{tocdepth}{1}}

In Conjectures \ref{conj1} and \ref{conj2}, we stated that the coefficients of the universal power series have $\C$-coefficients. The power series for which we provided an explicit formula in these conjectures have $\Q$-coefficients. For the remaining power series, studied in the previous sections, we found that their coefficients appear to lie in certain Galois extensions of $\Q$. This leads to further conjectures, which we will state for the Segre case, but which can be similarly formulated for the Verlinde case (Remark \ref{GalV}).

Consider the universal function of Conjecture \ref{conj1} 
\begin{align*}
\Phi:= &\rho^{2 - \chi+K^2} \, V_s^{c_2(\alpha)} W_s^{c_1(\alpha)^2} X_s^{\chi} e^{L^2 Q_s + (c_1(\alpha)L) R_s + u \, T_s} \sum_{J \subset [\rho-1]} (-1)^{|J| \chi}  \eps_{\rho}^{\|J\| K c_1} Y_{J,s}^{c_1(\alpha) K} Z_{J,s}^{K^2} e^{(K L) S_{J,s}}.
\end{align*}

\subsubsection*{$\boldsymbol{\rho=2}$.} For rank $\rho=2$ all universal power series of Section \ref{sec:alg:rk2} have rational coefficients and we conjecture this is always the case:
\begin{conjecture}
Let $\rho = 2$. Then $Y_{J,s},Z_{J,s},S_{J,s}  \in \Q[[z^{\frac{1}{2}}]]$ for all $J,s$. 
\end{conjecture}

\subsubsection*{$\boldsymbol{\rho=3}$.} Consider the following generators of $\mathrm{Gal}(\Q(i,3^{\frac{1}{2}}) / \Q) \cong \Z_2 \times \Z_2$
\begin{align*}
\sigma : 3^{\frac{1}{2}} \mapsto - 3^{\frac{1}{2}}, \quad i \mapsto i, \\
\tau : 3^{\frac{1}{2}} \mapsto 3^{\frac{1}{2}}, \quad i \mapsto -i.
\end{align*}
For any formal power series $F$ with $\Q(i,3^{\frac{1}{2}})$-coefficients, we denote by $\sigma(F)$ the power series obtained by acting by $\sigma \in \mathrm{Gal}(\Q(i,3^{\frac{1}{2}}) / \Q)$ on its coefficients. In accordance with the expressions of Section \ref{sec:alg:rk3}, we conjecture the following: 
\begin{conjecture} \label{conjGal2}
Let $\rho=3$. Then $Y_{J,s}, Z_{J,s}, S_{J,s} \in \Q(3^{\frac{1}{2}})[[z^{\frac{1}{2}}]]$ for $J = \varnothing, \{1,2\}$ and all $s \in \Z$  and $Y_{J,s}, Z_{J,s}, S_{J,s} \in \Q(3^{\frac{1}{2}}i)[[z^{\frac{1}{2}}]]$  for $j=\{1\},\{2\}$ and all $s \in \Z$. Moreover for any $L_{J,s} = Y_{J,s}, Z_{J,s}, S_{J,s}$ and $s \in \Z$, we have
\begin{align*}
\sigma(L_{\varnothing,s}) = L_{\{1,2\},s}, \quad \sigma(L_{\{1\},s}) = L_{\{2\},s}, \quad \tau(L_{\{1\},s}) = L_{\{2\},s}.
\end{align*}
\end{conjecture}

Assuming Conjecture \ref{conjGal2}, a direct calculation shows
$$
\sigma(\Phi) = \Phi, \quad \tau(\Phi) = \Phi.
$$
Hence $\Phi$ is invariant under $\mathrm{Gal}(\Q(i,3^{\frac{1}{2}}) / \Q)$. So for any $S,\alpha,L$, the power series $\Phi$ has \emph{rational} coefficients (as it should). This provides a consistency check on Conjecture \ref{conj1}. 

\subsubsection*{$\boldsymbol{\rho=4}$.} Consider the following generators of $\mathrm{Gal}(\Q(i,2^{\frac{1}{2}}) / \Q) \cong \Z_2 \times \Z_2$
\begin{align*}
\sigma : 2^{\frac{1}{2}} \mapsto -2^{\frac{1}{2}}, \quad i \mapsto i, \\
\tau : 2^{\frac{1}{2}} \mapsto 2^{\frac{1}{2}}, \quad i \mapsto -i.
\end{align*}
Based on the expressions of Section \ref{sec:alg:rk4}, we conjecture the following: 
\begin{conjecture} \label{conjGal3}
Let $\rho = 4$. Then $Y_{J,s}$, $Z_{J,s}$, $S_{J,s}  \in \Q(2^{\frac{1}{2}})[[z^{\frac{1}{2}}]]$ for   $\|J\|$ even and all $s \in \Z$, and $Y_{J,s}$, $Z_{J,s}$, $S_{J,s}  \in \Q(i)[[z^{\frac{1}{2}}]]$ for $\|J\|$ odd and all $s \in \Z$. Moreover for any $L_{J,s} = Y_{J,s}, Z_{J,s}, S_{J,s}$ and $s \in \Z$, we have
\begin{align*}
\sigma(L_{\varnothing,s})&=L_{\{1,3\},s},\  \sigma(L_{\{2\},s})=L_{\{1,2,3\},s},\ 
\tau(L_{\{1\},s})=L_{\{3\},s},\ \tau(L_{\{1,2\},s})=L_{\{2,3\},s}.
\end{align*}
\end{conjecture}
Assuming Conjecture \ref{conjGal3}, a direct calculation shows that
$$
\sigma(\Phi) = \Phi, \quad \tau(\Phi) = \Phi.
$$
Hence $\Phi$ is invariant under $\mathrm{Gal}(\Q(i,2^{\frac{1}{2}}) / \Q)$. As in the rank 3 case, for any $S,\alpha,L$, the power series $\Phi$ has rational coefficients---another consistency check of Conjecture \ref{conj1}. 

\begin{remark}
We found another interesting symmetry involving the formal parameter $t$ (defined via the change of variables $z = t (1+(1-\tfrac{s}{\rho}) t)^{1-\frac{s}{\rho}}$). Consider the two commuting involutions of $\C[[t^{\frac{1}{2}}]]$ determined by 
 $\upsilon:t^{\frac{1}{2}}\to -t^{\frac{1}{2}}$ and $\tau:i\to -i$. 
Based on our calculations, for $\rho=2,3,4$ and $L_{J,s}$ equal to $Y_{J,s},Z_{J,s},S_{J,s}$ for any $J \subset [\rho-1]$ and $s \in \Z$, we conjecture the relation
 $$\upsilon \tau(L_{J,s})=L_{[\rho-1]\setminus J,s}.$$ 
It seems natural to expect that this relation holds for all ranks $\rho$.
 \end{remark}
 
 \begin{remark} \label{GalV}
 We conjecture that the statements of this subsection hold (verbatim) on the Verlinde side with $\Phi$ replaced by the the universal function of Conjecture \ref{conj2}
\begin{align*}
\Psi := &\rho^{2 - \chi+K^2} \, G_{r}^{\chi(L)} F_{r}^{\frac{1}{2} \chi} \sum_{J \subset [\rho-1]} (-1)^{|J| \chi} \eps_{\rho}^{\|J\| K c_1} A_{J,r}^{K L} B_{J,r}^{K^2},
\end{align*}
and $Y_{J,s},Z_{J,s}$ replaced by $A_{J,r},B_{J,r}$.
 \end{remark}

\section{$K3$, Serre duality, Mari\~{n}o-Moore conjecture} 
\addtocontents{toc}{\protect\setcounter{tocdepth}{2}}

\subsection{$K3$ surfaces} 

Let $(S,H)$ be a polarized $K3$ surface and fix $\rho > 0$, $c_1,c_2$ such that $M:=M_S^H(\rho,c_1,c_2)$ only contains Gieseker $H$-stable sheaves. Then $M$ is deformation equivalent to $S^{[n]}$ \cite{OG,Huy, Yos}, where
$$
n := \tfrac{1}{2}\vd(M) = \rho c_2 - \tfrac{1}{2}(\rho-1) c_1^2 - (\rho^2-1).
$$
Let $\alpha \in K(S)$ be a class of rank $\rk(\alpha)$ and with Chern classes $c_1(\alpha)$ and $c_2(\alpha)$. We define
$$
\vd(\alpha) := 2\rk(\alpha) c_2(\alpha) - (\rk(\alpha)-1) c_1(\alpha)^2 - 2(\rk(\alpha)^2-1).
$$

It is natural to ask whether the insertion $c(\alpha_M)$ on $M$ deforms along to $S^{[n]}$. In this section, we describe a conjectural answer. Let $L \in \Pic(S)$ and let 
$$
\beta:=\beta(\tfrac{1}{\rho}\rk(\alpha),c_1(\alpha)^2, c_1(\alpha)L,\vd(\alpha)) \in K(S) \otimes \Q
$$
be a $K$-theory class satisfying
$$
\rk(\beta) = \tfrac{1}{\rho}\rk(\alpha), \quad c_1(\beta)^2 = c_1(\alpha)^2, \quad c_1(\beta) L = c_1(\alpha) L, \quad \mathrm{vd}(\beta) = \vd(\alpha).
$$
Such a class can be constructed from
$$
\tfrac{1}{\rho} \alpha + (1-\tfrac{1}{\rho}) ([\O_S(c_1(\alpha))] - [\O_S]) + c \cdot [\O_P] 
$$
for appropriate $c \in \Q$ and where $P \in S$. We conjecture the following:\footnote{Conjecture \ref{conjK3} was recently proved by Oberdieck \cite{Obe}.}
\begin{conjecture} \label{conjK3}
Let $(S,H)$ be a polarized K3 surface and fix $\rho>0$, $c_1,c_2$ such that $M:=M_S^H(\rho,c_1,c_2)$ only contains Gieseker $H$-stable sheaves. Then for any $\alpha \in K(S)$
$$
\int_{M} c(\alpha_M) e^{\mu(L)+\mu(\pt) u}= \int_{S^{[n]}} c(\beta^{[n]}) e^{\mu(L)+\rho \mu(\pt) u},
$$
where $n:=\frac{\vd(M)}{2}$ and $\beta:=\beta(\tfrac{1}{\rho}\rk(\alpha),c_1(\alpha)^2,c_1(\alpha)L, \vd(\alpha))  \in K(S) \otimes \Q$.
\end{conjecture}

\begin{corollary}
For $K3$ surfaces and $u=L=0$, Conjecture \ref{conjK3} implies Conjecture \ref{conj1strong}.
Moreover, for $K3$ surfaces, Conjecture \ref{conj1strong} for $\rho=1$ and  Conjecture \ref{conjK3} together imply Conjecture \ref{conj1strong} for all $\rho > 0$.\end{corollary}
\begin{proof}
The Segre numbers $\int_{S^{[n]}} c(\beta^{[n]})$ are calculated by \cite[Thm.~1]{MOP3} (Theorem \ref{MOPthm1}). The rank and Chern classes of $\beta$ are expressed in terms of the rank and Chern classes of $\alpha$ as follows
\begin{align*}
\rk(\beta) &= \tfrac{1}{\rho}\rk(\alpha), \quad c_1(\beta)^2 = c_1(\alpha)^2, \\
c_2(\beta) &= \rho c_2(\alpha) + \tfrac{1}{2}(1-\rho) c_1(\alpha)^2 - \rho \rk(\alpha) + \tfrac{1}{\rho} \rk(\alpha).
\end{align*}
The resulting expression is precisely the formula of Conjecture \ref{conj1strong} for a $K3$ surface and rank $\rho$, where we recall that the only Seiberg-Witten basic class of $S$ is $0$ and $\SW(0)=1$. This proves the first part of the corollary. The second part follows similarly.
\end{proof}

\begin{remark}
Let $S$ be a $K3$ surface with polarizations $H,H'$ and $L \in \Pic(S)$. Let $\rho, \rho' \in \Z_{>0}$ and $c_1,c_1',c_2,c_2'$ be such that $M_{S}^H(\rho,c_1,c_2)$, $M_S^{H'}(\rho',c_1',c_2')$ contain no strictly semistable sheaves. Let $\alpha,\alpha' \in K(S)$ such that
$$
\frac{\rk(\alpha)}{\rho} = \frac{\rk(\alpha')}{\rho'}, \quad c_1(\alpha)^2 = c_1(\alpha')^2, \quad c_1(\alpha) L = c_1(\alpha') L, \quad \vd(\alpha) = \vd(\alpha').
$$
Then Conjecture \ref{conjK3} implies
$$
\int_{M_S^H(\rho,c_1,c_2)} c(\alpha_M) e^{\mu(L)+ \mu(\pt) u / \rho} = \int_{M_{S}^{H'}(\rho',c_1',c_2')} c(\alpha'_M) e^{\mu(L)+ \mu(\pt) u / \rho'}.
$$
\end{remark}

\subsection{Virtual Serre duality}

Applying virtual Serre duality \cite[Prop.~3.13]{FG} to the virtual Verlinde numbers discussed in the introduction gives 
\begin{align}
\begin{split} \label{vSD}
\chi^{\vir}(M, \mu(L) \otimes E^{\otimes r}) &= (-1)^{\vd(M)} \chi^{\vir}(M, \mu(-L) \otimes E^{\otimes -r} \otimes K_{M}^{\vir}) \\
&= (-1)^{\vd(M)} \chi^{\vir}(M, \mu(-L+\rho K_S) \otimes E^{\otimes -r}),
\end{split}
\end{align}
where $K_M^{\vir} := \Lambda^{\vd(M)} \Omega_M^{\vir}$ and the second equality follows from $c_1(T_M^{\vir}) = - \rho \mu(K_S)$ \cite[Prop.~8.3.1]{HL}. This gives relations among the coefficients of the universal functions of Theorem \ref{EGLthm} and Conjecture \ref{conj2}. Clearly, by \eqref{fg}, we have
\begin{align*}
f_{-r/\rho} = f_{r/\rho}, \quad g_{-r/\rho} = g_{r/\rho},
\end{align*}
for all $\rho, r$. Furthermore, for $\rho=1$ (see also \cite{EGL})
$$
A_r = \frac{B_{-r}}{B_r}.
$$
for all $r \in \Z$ (where we dropped the subscript $J = \varnothing$). In general, using the fact that $aK = a^2$ for Seiberg-Witten basic classes $a \in H^2(S,\Z)$, virtual Serre duality \eqref{vSD} suggests the following relations:
\begin{conjecture} \label{conj:SDvir}
For any $\rho > 0$, we have
\begin{align*}
A_{J,-r}(w^{\frac{1}{2}})&=g_{r/\rho}(w)^{1-\rho} A_{J,r}(-w^{\frac{1}{2}})^{-1}, \\
B_{J,-r}(w^{\frac{1}{2}})&=g_{r/\rho}(w)^{\binom{\rho}{2}}A_{J,r}(-w^{\frac{1}{2}})^\rho B_{J,r}(-w^{\frac{1}{2}}),
\end{align*}
for all $J \subset [\rho-1]$ and $r \in \Z$.
\end{conjecture}

In particular, the universal power series $A_{J,r}, B_{J,r}$ with $r<0$ are determined by the universal power series with $r>0$ (and vice versa).

We can combine virtual Serre duality and the virtual Segre-Verlinde correspondence in order to obtain interesting relations among the universal power series for the virtual Segre numbers. A direct calculation shows the following:

\begin{corollary} \label{cor:SDvir}
Assuming Conjectures \ref{conj3} and \ref{conj:SDvir}, we obtain
\begin{align*}
Y_{J,2\rho-s}(z^{\frac{1}{2}})&=\frac{(1+(1-\frac{s}{\rho}) \tau)^\rho}{(1+2(1-\frac{s}{\rho}) \tau)^{\tfrac{\rho}{2}}} Y_{J,s}(-\zeta^{\frac{1}{2}})^{-1},\\
Z_{J,2\rho-s}(z^{\frac{1}{2}})&= \frac{(1+(2-\tfrac{s}{\rho})\tau)^{\frac{1}{2}\rho(\rho-s)}}{(1+(1-\tfrac{s}{\rho})\tau)^{\frac{1}{2} \rho(2\rho-s)}} Y_{J,s}(-\zeta^{\frac{1}{2}})^\rho Z _{J,s}(-\zeta^{\frac{1}{2}}), 
\end{align*}
for all $J \subset [\rho-1]$ and $s \in \Z$, where $z = t (1-(1-\frac{s}{\rho}) t)^{-1+\frac{s}{\rho}}$, $\zeta = \tau (1+(1-\frac{s}{\rho}) \tau)^{1-\frac{s}{\rho}}$, and $\tau =t(1-2(1-\smfr{s}{\rho})t)^{-1}$.
\end{corollary}

In particular, the universal power series $Y_{J,s}, Z_{J,s}$ with $s<\rho$ are determined by the universal power series with $s>\rho$ (and vice versa). 
In all the cases where we conjectured explicit formulae for the universal power series $A_{J,r}, B_{J,r}$ and $Y_{J,s}, Z_{J,s}$ in Section \ref{sec:alg}, we verified that they satisfy the formulae of Conjecture \ref{conj:SDvir} and Corollary \ref{cor:SDvir}.

\subsection{Mari\~{n}o-Moore conjecture} \label{sec:Donaldson}

Donaldson invariants \cite{Don} are diffeomorphism invariants of differentiable $4$-manifolds.
Let $(S,H)$ be a smooth polarized surface satisfying $b_1(S) = 0$, $p_g(S)>0$, and let $L \in \Pic(S)$. Suppose $M:=M_S^H(\rho,c_1,c_2)$ does not contain strictly semistable sheaves. In algebraic geometry, one can define the corresponding $\SU(\rho)$ Donaldson invariants of $S$ by
\begin{align*}
D_{\rho,c_1,c_2}^{S,H}(L  + \pt \, u) &:= \int_{[M^{H}_S(\rho,c_1,c_2)]^\vir} e^{\mu(L)+\mu(\pt)u}, \\
D_{\rho,c_1}^{S,H}(L  + \pt \, u) &:= \sum_{c_2} D_{\rho,c_1,c_2}^{S,H}(L  + \pt \, u) \, z^{\frac{\vd}{2}}
\end{align*}
where $\pt \in H^4(S,\Z)$ denotes the (Poincar\'e dual of the) point class, $u$ is a formal variable, and $\vd$ is given by \eqref{vd}. As mentioned in Section \ref{sec:moc}, $S$ has finitely many Seiberg-Witten basic classes $a \in H^2(S,\Z)$ satisfying $a^2 = a K$ and $\SW(K-a) = (-1)^\chi \SW(a)$. When $a$ is a Seiberg-Witten basic class of $S$, we define
$$
\widetilde{a} := 2a - K.
$$
Then $\widetilde{a}$ corresponds to a Seiberg-Witten basic class in the sense of differential geometry and $\SW(a) = \widetilde{\SW}(\widetilde{a})$, where $\widetilde{\SW}$ denotes the Seiberg-Witten invariant from differential geometry \cite{Mor}.

In the rank 1 case, Proposition \ref{rk1Witten} gives 
\begin{equation} \label{rk1D}
D_{1,c_1}^{S,H}(L+\pt \, u,z)=e^{(\frac{1}{2}L^2+u)z}.
\end{equation}
The Witten conjecture \cite{Wit} for the $\SU(2)$ Donaldson invariants of $S$ can be stated as follows: $D_{2,c_1,c_2}^{S,H}(L  + \pt \, u)$ equals the coefficient of $z^{\mathrm{vd}(M)/2}$ of
\begin{equation} \label{rk2D}
2^{2-\chi+K^2}\ e^{(\frac{1}{2} L^2+2 u) z}\sum_{a \in H^2(S,\Z)} \SW(a) (-1)^{a c_1} e^{-\widetilde{a} L z^{\frac{1}{2}}}.
\end{equation}
This was proved in \cite{GNY3} for complex smooth projective surfaces with $b_1(S) = 0$ and $p_g(S)>0$, and (under a technical assumption) for all differentiable $4$-manifolds $M$ with $b_1(M) = 0$, odd $b^+(M)$, and of Seiberg-Witten simple type in \cite{FL1, FL2}.
In \cite{MM}, M.~Mari\~{n}o and G.~Moore predicted the existence of $\SU(\rho)$ Donaldson invariants and a higher rank Witten conjecture was proposed. See also \cite[(10.107)]{LM}. The gauge theoretic definition of higher rank Donaldson invariants was given by P.B.~Kronheimer \cite{Kro}. 

In the special case $\alpha=0$, Conjecture \ref{conj1strong} implies a structure formula for the higher rank Donaldson invariants of $S$. Then $s=0$ and we have universal power series $X_0, Q_0, T_0$, $Z_0, Z_{jk,0}$, $S_0, S_{j,0}$, for all $1 \leq j \leq k \leq \rho-1$, satisfying
$$
X_0(z) = 1, \quad Q_0(z)=\tfrac{1}{2} z, \quad T_0(z)=\rho z.
$$
Moreover, in Section \ref{sec:verif}, we conjecturally found for $\rho=2,3,4$ that $Z_0, Z_{jk,0}$ are certain constants and $S_0, S_{j,0}$ are certain multiples of $z^{\frac{1}{2}}$. This leads us to the following explicit ``$\SU(3)$ and $\SU(4)$ Witten conjectures''. 

Let $(S,H)$ be a smooth polarized surface with $b_1(S) = 0$, $p_g(S)>0$, and let $L \in \Pic(S)$. Suppose $M:=M_S^H(\rho,c_1,c_2)$ does not contain strictly semistable sheaves. Then, for $\rho=3$ and $\rho=4$, the Donaldson invariant  $D_{\rho,c_1,c_2}^{S,H}(L  + \pt \, u)$ is given by the coefficient of $z^{\mathrm{vd}(M)/2}$ of (respectively)
\begin{align}
\begin{split} \label{rk3and4D}
3^{2-\chi+K^2} e^{(\frac{1}{2}L^2+3 u)z} \sum_{(a_1,a_2)} &\SW(a_1) \SW(a_2) \eps_3^{(a_1+2a_2)c_1} e^{\frac{\sqrt{3}}{2}(\widetilde a_1+\widetilde a_2)Lz^{\frac{1}{2}}}2^{\frac{1}{4}(\widetilde a_1+\widetilde a_2)^2},\\
4^{2-\chi+K^2} e^{(\frac{1}{2}L^2+4 u)z}\sum_{(a_1,a_2,a_3)} & \SW(a_1) \SW(a_2) \SW(a_3) i^{(-a_1+2a_2-3a_3)c_1} e^{\big(\frac{\sqrt{2}}{2}\widetilde a_1-\widetilde a_2+\frac{\sqrt{2}}{2}\widetilde a_3 \big)Lz^{\frac{1}{2}}} \\
&\times 2^{K^2+ \frac{1}{4}(\widetilde a_1 + \widetilde a_2)(\widetilde a_1+\widetilde a_3)} \big(1-\smfr{\sqrt{2}}{2} \big)^{\frac{1}{2}\widetilde a_2(\widetilde a_1 +\widetilde a_3)}.
\end{split}
\end{align}
In the rank 4 case, we used some slight rewriting involving the relations $a^2 = a K$ and $\SW(K-a) = (-1)^\chi \SW(a)$ for Seiberg-Witten basic classes $a \in H^2(S,\Z)$.

Based on this, it is natural to conjecture the following, which can be seen as an algebro-geometric version of the Mari\~{n}o-Moore conjecture \cite[(9.17)]{MM}, \cite[(10.107)]{LM}.
\begin{conjecture} \label{Donaldsonconj}
For any $\rho \in \Z_{>0}$, there  are universal algebraic numbers $\gamma$, $\gamma_{jk}$,  $\delta_\ell \in \overline \Q$ for all $j <k,\ell \in \{ 1, \ldots, \rho-1 \}$ with the following property. Let $(S,H)$ be a smooth polarized surface with $b_1(S) = 0$, $p_g(S)>0$, and let $L \in \Pic(S)$. Suppose $M:=M_S^H(\rho,c_1,c_2)$ does not contain strictly semistable sheaves. Then $D_{\rho,c_1,c_2}^{S,H}(L  + \pt \, u)$ equals the coefficient of $z^{\mathrm{vd}(M)/2}$ of
\begin{align*}
\rho^{2-\chi+K^2} \gamma^{K^2}  e^{(\frac{1}{2}L^2+\rho u)z}\sum_{(a_1,\ldots,a_{\rho-1})} \prod_{j=1}^{\rho-1} \eps_{\rho}^{j(a_jc_1)} \SW(a_j) e^{\delta_j(\widetilde a_j L) z^{\frac{1}{2}}} \prod_{1\le j< k\le \rho-1}\gamma_{jk}^{\widetilde a_j\widetilde a_k},
\end{align*}
where the sum is over all $(a_1, \ldots, a_{\rho-1}) \in H^2(S,\Z)^{\rho-1}$.
\end{conjecture}

\begin{remark}
For $c_1=0$, the previous expressions manifestly only depend on the oriented diffeomorphism type of $S$ (by expression $\chi$, $K^2$ in terms of $e(S)$, $\sigma(S)$). In this case, the ``stable equals semistable'' assumption can be satisfied by taking $\gcd(\rho,c_2)=1$. If in addition $\vd(M) = 0$, we obtain intriguing expressions for the ``virtual point count'' $\int_{[M]^{\vir}} 1 = e^{\vir}(M) \in \Z$.
\end{remark}

\begin{remark}
For fixed $\rho,S,H,L,c_1$, the power series in Conjecture \ref{Donaldsonconj} is (a priori) an element of $\overline{\Q}[[u, z^{\frac{1}{2}}]]$. From the definition of Donaldson invariants, we also know it should be an element of $\Q[[u, z^{\frac{1}{2}}]]$. For ranks $\rho=1$ and $\rho=2$, note that \eqref{rk1D} and \eqref{rk2D} manifestly have rational coefficients. However, for $\rho = 3$ and $\rho = 4$, the power series in \eqref{rk3and4D} a priori are elements of $\Q(i,3^{\frac{1}{2}})[[u,z^{\frac{1}{2}}]]$ and $\Q(i,2^{\frac{1}{2}})[[u, z^{\frac{1}{2}}]]$. If $S$ is minimal of general type, then the Galois invariance discussed in Section \ref{sec:Gal} shows that they indeed have rational coefficients.
\end{remark}

From the perspective of Donaldson theory and Nekrasov partition functions, Conjecture \ref{conj1strong}, and the corresponding algebraic functions in Section \ref{sec:alg}, can be seen as a Witten type conjecture for gauge group $\SU(\rho)$ and fundamental matter insertions.

\subsection{Disconnected canonical curve}

\begin{proposition}
Let $S$ be a smooth projective surface such that $b_1(S) = 0$, $p_g(S)>0$, and $|K_S|$ contains an element of the form $C_1 + \cdots +C_m$, where $C_j \subset S$ are mutually disjoint irreducible reduced curves. Fix $H,\rho>0,c_1,c_2$ such that $M:=M_S^H(\rho,c_1,c_2)$ contains no strictly semistable sheaves. Let $L \in \Pic(S)$, $r \in \Z$, and assume Conjecture \ref{conj2strong} holds for $\chi^{\vir}(M,\mu(L) \otimes E^{\otimes r})$. Then it is given by the coefficient of $w^{\vd(M)/2}$ of
$$
\rho^{2 - \chi+K^2} \, G_r^{\chi(L)} F_r^{\frac{1}{2} \chi}  \prod_{\ell =1}^{m}  \sum_{J \subset [\rho - 1]} (-1)^{|J| h^0(N_{C_\ell / S})} \eps_{\rho}^{\|J\| C_\ell c_1} A_{J,r}^{L C_\ell} B_{J,r}^{C_\ell^2},
$$ 
where $N_{C_\ell / S}$ denotes the normal bundle of $C_\ell$ in $S$.
\end{proposition}
\begin{proof}
For any $J \subset [m] := \{1, \ldots, m\}$, define $C_J := \sum_{j \in J} C_j$, where $C_{\varnothing} := 0$. For any $J,K \subset [m]$, we define $J \sim K$ when $C_J$ and $C_K$ are linearly equivalent. This defines an equivalence relation and we denote the equivalence class corresponding to $J$ by $[J]$. By \cite[Lem.~5.14]{GK1}, the Seiberg-Witten basic classes of $S$ are $\{C_J\}_{[J]}$, $\SW(0)=1$, and
$$
\SW(C_J) = |[J]| \prod_{j \in J} (-1)^{h^0(N_{C_j / S})}, \quad \forall \varnothing \neq J \subset [m].
$$
By Conjecture \ref{conj2}, $\chi^{\vir}(M,\mu(L) \otimes E^{\otimes r})$ equals the coefficient of $w^{\vd(M)/2}$ of
\begin{align*}
\rho^{2 - \chi+K^2} \, G_{r}^{\chi(L)} F_{r}^{\frac{1}{2} \chi} A_{r}^{L K} B_{r}^{K^2} \sum_{(J_1, \ldots, J_{\rho-1}) \atop J_1, \ldots, J_{\rho-1} \subset [m]} \prod_{j=1}^{\rho-1} \eps_{\rho}^{j C_{J_j} c_1} (-1)^{h^0(N_{C_{J_j /S}})} A_{j,r}^{L C_{J_j}} \prod_{1 \leq j \leq k \leq \rho-1} B_{jk,r}^{C_{J_j} C_{J_k}}.
\end{align*}
For any choice of $J_1, \ldots, J_{\rho-1} \subset [m]$ one obtains a partitioning of $[m]$ as follows. Define $P_{[\rho-1]} := \bigcap_{j \in \{1, \ldots, \rho - 1\}} J_j$. Then, for any $J \subsetneq [\rho-1]$, we inductively define
\begin{align*}
P_J &:= \Big( \bigcap_{j \in J} J_j \Big) \setminus \bigcup_{J \subsetneq K} P_K \\
&\cdots \\
P_{\varnothing} &:= [m] \setminus \bigcup_{ \varnothing \neq K} P_K.
\end{align*}
By inclusion-exclusion, we have 
$$
[m] = \bigsqcup_{J \subset [\rho-1]} P_J.
$$
We can therefore replace the sum $\sum_{(J_1, \ldots, J_{\rho-1})}$ above by the sum
$$
\sum_{[m] = \bigsqcup_{J \subset [\rho-1]} P_J}
$$
over all partitionings of $[m]$. Writing the summands in terms of the $P_J$, we obtain
$$
\rho^{2 - \chi+K^2} \, G_r^{\chi(L)} F_r^{\frac{1}{2} \chi} A_{r}^{LK} B_{r}^{K^2} \prod_{\ell =1}^{m} \sum_{J \subset [\rho - 1]} \prod_{j \in J} ((-1)^{h^0(N_{C_\ell / S})} \eps_{\rho}^{j C_\ell c_1} A_{j,r}^{L C_\ell}) \prod_{(j\leq k) \in J \times J} B_{jk,r}^{C_\ell^2},
$$ 
where we define the product over $J = \varnothing$ to be equal to 1. The result follows from the definition of the power series $A_{J,r}, B_{J,r}$ in \eqref{Jseries}.
\end{proof}

\subsection{Blow-up formula}

\begin{proposition}
Let $S$ be a smooth projective surface satisfying $b_1(S)=0$ and $p_g(S)>0$. Let $\rho \in \Z_{>0}$, $c_1 \in H^2(S,\Z)$, $L \in \Pic(S)$, and $r \in \Z$. Denote formula \eqref{Verlindenum} by $\psi_{S,\rho,c_1,L,r}$. Let $\pi : \widetilde{S} \rightarrow S$ be the blow-up of $S$ in one point with exceptional divisor $E$ and set $\widetilde{c}_1 = \pi^* c_1 - m E$, and $\widetilde{L} = \pi^* L - \ell E$. Then  
$$
\psi_{\widetilde{S},\rho,\widetilde{c}_1,\widetilde{L},r} = \frac{1}{\rho G_r^{\binom{\ell+1}{2}}} \Bigg[ \sum_{J \subset [\rho-1]}  \eps_{\rho}^{\|J\| m} \frac{A_{J,r}^{\ell}}{B_{J,r}} \Bigg] \psi_{S,\rho,c_1,L,r}.
$$
\end{proposition}
\begin{proof}
By \cite[Thm.~7.4.6]{Mor}, the set of Seiberg-Witten basic classes of $\widetilde{S}$ is given by $\{\pi^*a, \pi^*a + E\}$, where $a$ runs over all Seiberg-Witten basic classes of $S$, and 
$$
\SW(\pi^* a) = \SW(\pi^*a + E) = \SW(a).
$$
The result now follows from an easy calculation using $K_{\widetilde{S}} = \pi^* K_S +E$, $\chi(\O_{\widetilde{S}}) = \chi(\O_S)$, and $E^2=-1$.
\end{proof}

{\tt{gottsche@ictp.it, m.kool1@uu.nl}}
\end{document}